\documentclass[french,english,reqno,twoside,a4paper]{amsart}
\newcommand*{\letitre}{Differential algebras of quasi-Jacobi forms of index zero}
\newcommand*{\lesmotscles}{Elliptic forms, Jacobi forms, formal deformations, Rankin-Cohen brackets}
\newcommand*{\lescodesMSC}{Primary: 11F50, 16S80. Secondary: 11F11, 11F25, 16W25, 53D55}
\newcommand*{\lesujet}{Math. Subj. Class~\lescodesMSC}
\newcommand{\leresume}{The notion of double depth associated with quasi-Jacobi forms allows distinguishing, within the algebra \(\AlgQJS\) of quasi-Jacobi singular forms of index zero, certain significant subalgebras (modular-type forms, elliptic-type forms, Jacobi forms). We study the stability of these subalgebras under the derivations of \(\AlgQJS\) and through certain sequences of bidifferential operators constituting analogs of Rankin-Cohen brackets or transvectants.}
\usepackage[utf8]{inputenc} 
\usepackage[T1]{fontenc}
\usepackage[narrowiints,partialup]{kpfonts}
\usepackage{xcolor}
\usepackage{babel}
\usepackage[all]{xy}
\usepackage[showonlyrefs]{mathtools}
\usepackage[useregional]{datetime2}
\usepackage[paper=a4paper]{geometry}
\usepackage{varioref,hyperref}
\usepackage{microtype}
\usepackage[all]{nowidow}
\usepackage{enumitem,placeins,tabularx,tikz,rotating}
\usetikzlibrary{calc}
\newcommand*\texten[1]{#1}
\makeatletter
 \newlength{\h@uteurnumerateur}
 \newlength{\h@uteurdenominateur}
\newcommand*{\qd}[2]{
  \mathchoice%
  {
    \settoheight{\h@uteurnumerateur}{\ensuremath{\displaystyle{#1#2}}}%
    \settoheight{\h@uteurdenominateur}{\ensuremath{\displaystyle{#1#2}}}%
    \raisebox{0.5\h@uteurnumerateur}{\ensuremath{\displaystyle{#1}}}%
    \mkern-5mu\diagup\mkern-4mu%
    \raisebox{-0.5\h@uteurdenominateur}{\ensuremath{\displaystyle{#2}}}%
  }
  {
    \settoheight{\h@uteurnumerateur}{\ensuremath{\textstyle{#1#2}}}%
    \settoheight{\h@uteurdenominateur}{\ensuremath{\textstyle{#1#2}}}%
  \raisebox{0.2\h@uteurnumerateur}{\ensuremath{\textstyle{#1}}}%
  /%
  \raisebox{-0.2\h@uteurdenominateur}{\ensuremath{\textstyle{#2}}}%
  }
  {
    \settoheight{\h@uteurnumerateur}{\ensuremath{\scriptstyle{#1#2}}}%
    \settoheight{\h@uteurdenominateur}{\ensuremath{\scriptstyle{#1#2}}}%
  \raisebox{0.2\h@uteurnumerateur}{\ensuremath{\scriptstyle{#1}}}%
  /%
  \raisebox{-0.2\h@uteurdenominateur}{\ensuremath{\scriptstyle{#2}}}%
  }
  {
    \settoheight{\h@uteurnumerateur}{\ensuremath{\scriptscriptstyle{#1#2}}}%
    \settoheight{\h@uteurdenominateur}{\ensuremath{\scriptscriptstyle{#1#2}}}%
  \raisebox{0.2\h@uteurnumerateur}{\ensuremath{\scriptscriptstyle{#1}}}%
  /%
  \raisebox{-0.2\h@uteurdenominateur}{\ensuremath{\scriptscriptstyle{#2}}}%
  }
}
\hypersetup{%
   unicode=true,
   hidelinks,%
   pdftitle={\letitre},%
   pdfauthor={\textcopyright\ Fran{ç}ois Dumas, Fran{ç}ois Martin et Emmanuel Royer},%
   pdfsubject={\lesujet},%
   pdfkeywords={\lesmotscles},%
   baseurl={https://royer.perso.math.cnrs.fr},%
   pdflang=fr-FR,
  }
\DeclarePairedDelimiter{\abs}{\lvert}{\rvert}
\DeclarePairedDelimiterXPP\accol[3]{}\{\}{_{#3}}{
	\ifblank{#1}{\phantom{f}}{#1},
	\ifblank{#2}{\phantom{f}}{#2}
	}
\DeclareMathOperator{\AC}{H}
\newcommand*{\AlgJS}{\mathrm{JS}}
\newcommand*{\AlgM}{\mathrm{M}}
\newcommand*{\AlgQJS}{\mathrm{JS}^{\infty}}
\newcommand*{\AlgQJell}{\mathrm{JS}^{0,\infty}}
\newcommand*{\AlgQJmod}{\mathrm{JS}^{\infty,0}}
\newcommand*{\AlgQM}{\mathrm{M}^{\infty}}
\newcommand*{\CC}{\mathbb{C}}
\DeclareMathOperator{\CM}{CM}
\DeclareMathOperator{\cQ}{Q}
\DeclarePairedDelimiterXPP\Crochet[3]{}\llbracket\rrbracket{_{#3}}{
	\ifblank{#1}{\phantom{f}}{#1},
	\ifblank{#2}{\phantom{f}}{#2}
	}
\DeclarePairedDelimiterXPP\crochet[3]{}\lbrack\rbrack{_{#3}}{
	\ifblank{#1}{\phantom{f}}{#1},
	\ifblank{#2}{\phantom{f}}{#2}
}
\DeclarePairedDelimiter{\projZ}{\lVert}{\rVert}
\newcommand*{\dtau}{\mathop{\partial}\nolimits_{\tau}}
\newcommand*{\dz}{\mathop{\partial}\nolimits_{z}}
\DeclareMathOperator{\fe}{e}
\DeclareMathOperator{\eisen}{e}
\DeclareMathOperator{\Eisen}{E}
\DeclareMathOperator{\fd}{d}
\newcommand*{\fds}{\fd_{\mathrm{S}}}
\newcommand*{\fdmod}{\fd^{\infty,0}_{\mathrm{S}}}
\newcommand*{\fdell}{\fd^{0,\infty}_{\mathrm{S}}}
\newcommand*{\fdq}{\fd^{\infty}_{\mathrm{S}}}
\newcommand*{\fonc}{{\CC^\pkC}}
\DeclareMathOperator{\ia}{I}
\newcommand*{\ic}{\mathrm{i}}
\DeclareMathOperator{\J}{J}
\newcommand*{\Jac}{{\sldz\ltimes\Z^2}}
\newcommand*{\jc}{\mathrm{j}}
\newcommand*{\JS}[1]{\mathrm{JS}_{#1}}
\newcommand*{\JSell}[1]{\mathrm{JS}^{0,\infty}_{#1}}
\newcommand*{\JSmod}[1]{\mathrm{JS}^{\infty,0}_{#1}}
\newcommand*{\N}{\mathbb{Z}_{\geq0}}
\newcommand{\Ob}{\mathop{\mathrm{Ob}}\nolimits^*}
\DeclareMathOperator{\ObCDMR}{\mathrm{Ob}}
\DeclarePairedDelimiterXPP\parent[3]{}\lparen\rparen{_{#3}}{
	\ifblank{#1}{\phantom{f}}{#1},
	\ifblank{#2}{\phantom{f}}{#2}
	}
\DeclareMathOperator{\pa}{P}
\newcommand*{\pk}{\mathcal{H}}
\newcommand*{\pkC}{{\pk\times\CC}}
\newcommand*{\planck}{\hbar}
\newcommand*{\QJS}[3]{\mathrm{JS}_{#1}^{\leq#2,#3}}
\newcommand*{\QJSpoids}[1]{\mathrm{JS}_{#1}^{\infty}}
\newcommand{\Reseau}{\mathcal{R}}
\newcommand*{\Sing}{\mathcal{S}}
\newcommand*{\SL}{\mathrm{SL}}
\newcommand*{\sldz}{\SL(2,\Z)}
\DeclareMathOperator{\wpdp}{W}
\DeclareMathOperator{\X}{X}
\DeclareMathOperator{\Y}{Y}
\newcommand*{\Z}{\mathbb{Z}}
\def\l@paragraph{\@tocline{4}{0pt}{5pc}{7pc}{}}
\@xp\gdef\csname r@tocindent4\endcsname{0pt}
\def\paragraph{\@startsection{paragraph}{4}%
  \z@{.5\linespacing\@plus.7\linespacing}{-.5em}%
  {\normalfont\itshape}}
\makeatother
\theoremstyle{plain}
\newtheorem{thm}{Theorem}
\newtheorem{lem}[thm]{Lemma}
\newtheorem{prop}[thm]{Proposition}
\newtheorem{cor}[thm]{Corollary}
\theoremstyle{definition}
\newtheorem{dfn}[thm]{Definition}
\theoremstyle{remark}
\newtheorem{rem}[thm]{Remark}
\newcommand\tikzmark[2]{%
\tikz[remember picture,baseline] \node[inner sep=2pt,outer sep=0] (#1){#2};%
}
\newcommand\link[2]{%
\begin{tikzpicture}[remember picture, overlay, >=stealth, shift={(0,0)}]
  \draw[->] (#1) to (#2);
\end{tikzpicture}%
}
\newcommand*{\rose}[1]{{#1}}
\newcommand*{\bleu}[1]{{#1}}
\newcounter{cellindex}

\newcolumntype{C}{>{\stepcounter{cellindex}%
\begin{tikzpicture}[remember picture,baseline=(Cell-\thecellindex-left.base),inner sep=0pt]
\node (Cell-\thecellindex-left){\strut};
\end{tikzpicture}
}c<{\begin{tikzpicture}[remember picture,baseline=(Cell-\thecellindex-right.base),inner sep=0pt]
\node (Cell-\thecellindex-right){\strut};
\end{tikzpicture}}} 

\newcommand{\MeasureLastTable}[1]{
\pgfmathtruncatemacro{\LeftCellIndex}{\thecellindex-1}
\pgfmathtruncatemacro{\AboveCellIndex}{\thecellindex-#1}
\begin{tikzpicture}[overlay,remember picture]
    \path let \p0 = (Cell-\thecellindex-right.east), 
    \p1 = (Cell-\thecellindex-left.west),
    \p2 = (Cell-\LeftCellIndex-right.east),
    \p3 = (Cell-\LeftCellIndex-left.west) in 
    \pgfextra{\pgfmathsetmacro{\tmp}{(\x0+\x1-\x2-\x3)/4}
    \xdef\HorSep{\tmp pt} 
    };
    \path let \p0 = (Cell-\thecellindex-right.north west), 
    \p1 = (Cell-\thecellindex-left.south west),
    \p2 = (Cell-\AboveCellIndex-right.north west),
    \p3 = (Cell-\AboveCellIndex-left.south west) in 
    \pgfextra{\pgfmathsetmacro{\tmp}{(\y2+\y3-\y0-\y1)/4}
    \xdef\VertSep{\tmp pt} 
    }; 
\end{tikzpicture}%
}

\newcommand{\CrossOut}[2][]{
\tikz[overlay,remember picture]{
\path let \p0 = (Cell-#2-left.west), 
    \p1 = (Cell-#2-right.east),
    \p2 = (Cell-#2-left.north west),
    \p3 = (Cell-#2-left.south west) in
    \pgfextra{\pgfmathsetmacro{\tmp}{\HorSep-(\x1-\x0)/2}
    \xdef\myxoffset{\tmp pt}
    \pgfmathsetmacro{\tmp}{\VertSep-(\y2-\y3)/2}
    \xdef\myyoffset{\tmp pt}
    };
\draw[thick,#1] ([xshift=-\myxoffset,yshift=-\myyoffset]Cell-#2-left.south west) -- 
([xshift=\myxoffset,yshift=\myyoffset]Cell-#2-right.north east);
}}
\begin{document}
\title{\letitre}
\author[F. Dumas]{Fran{ç}ois Dumas}
\address{%
Fran{ç}ois Dumas\\
Université Clermont Auvergne -- CNRS\\
Laboratoire de mathématiques Blaise Pascal -- UMR6620\\
F-63000 Clermont-Ferrand\\
France
}
\email{{francois.dumas@uca.fr}}
\author[F. Martin]{Fran{ç}ois Martin}
\address{%
Fran{ç}ois Martin\\
Université Clermont Auvergne -- CNRS\\
Laboratoire de mathématiques Blaise Pascal -- UMR6620\\
F-63000 Clermont-Ferrand\\
France
}
\email{{francois.martin@uca.fr}}
\author[E. Royer]{Emmanuel Royer}
\address{%
Emmanuel Royer\\
Université Clermont Auvergne -- CNRS -- Wolfgang Pauli Institute\\
Institut CNRS Pauli -- IRL2842 \\
A-1090 Wien\\
Autriche
}
\address{%
Emmanuel Royer\\
CNRS – Université de Montréal CRM – CNRS\\
IRL3457\\
Montréal\\
Canada
}
\email{{emmanuel.royer@math.cnrs.fr}}
\date{\DTMnow}
\thanks{The work of the last two authors is partially funded by the ANR-23-CE40-0006-01 Gaec project. The third author benefited from valuable discussions with Yuk-kam Lau and Ben Kane during visits to the \emph{Institute of Mathematical Research} at \emph{Hong Kong University}, made possible by the Hubert Curien Procore project \no~48166WK. He thanks Timothy Browning from the \emph{Institute of Science and Technology Austria} for his hospitality. For the purpose of open access publication, this version of the text is distributed under the \href{https://creativecommons.org/licenses/by/4.0/deed.en}{CC-BY} license.}
\keywords{\lesmotscles}
\subjclass[2020]{\lescodesMSC}
\begin{abstract}
\leresume
\begin{center}
\textbf{This text is essentially an automatic translation with the assistance of \emph{ChatGPT 4o} from the original French author last version.}
\end{center}
\end{abstract}

\maketitle
\tableofcontents
\newcommand{\lintro}{Introduction}
\section{\lintro}
This article presents an analytical and algebraic study of singular quasi-Jacobi forms of index zero. It particularly examines the stability under derivations of certain significant subalgebras (elliptic forms, quasi-Jacobi forms of quasielliptic type, quasi-Jacobi forms of quasimodular type), with the aim of constructing sequences of bidifferential operators that constitute formal deformations of these algebras, namely, Rankin-Cohen brackets or transvectants.

For the actions (parameterized by a nonnegative integer, the weight) of the modular group \(\sldz\) on the algebra of functions of a complex variable \(\tau\) in the Poincaré half-plane \(\pk\) with values in \(\CC\), it is well known that the algebra \(\AlgM\) of modular forms (graded by weight) is not stable under the derivation \(\partial_\tau\). There are at least two ways to overcome this obstruction. The first is to canonically construct a sequence of bidifferential operators in \(\partial_\tau\), known as Rankin-Cohen brackets, which stabilize \(\AlgM\) (cf. \cite{zbMATH05808162}) and which also constitute (cf. \cite{zbMATH05156388}, \cite{zbMATH07362171}, and \cite{zbMATH05808162}) a formal deformation of the algebra \(\AlgM\) (in the sense of \cite[\texten{Chapter}~13]{zbMATH06054532}). The second is to define above \(\AlgM\) the algebra \(\AlgQM\) of quasimodular forms, which is by construction stable under \(\partial_\tau\), graded by weight, and filtered by depth (cf. \cite{zbMATH05808162}, \cite{zbMATH06128504}). 
These two points of view are closely related since one method to show the stability of \(\AlgM\) by Rankin-Cohen brackets involves extending their definition to the algebra \(\AlgQM\) (see \cite[\texten{Section}~5]{MR1280058} or \cite[\texten{Proposition}~9]{zbMATH07362171}). A similar approach is proposed in this article for the action of the Jacobi group on functions in two variables. It requires revisiting various notions scattered throughout the literature on Jacobi forms and quasi-Jacobi forms in a formalized and unified context (see for example \cite{MR4281261}, \cite{zbMATH05953688}, \cite{MR0781735}, \cite{Fogliasso}).

In what follows, we consider the actions (parameterized by two nonnegative integers, the weight and the index) of the Jacobi group \(\Jac\) on functions of two complex variables \((\tau,z)\) from \(\pkC\) to \(\CC\). The notion of a singular Jacobi form follows from this (definition~\ref{eq_dfnJacobiSing}), with the term singular referring here to the analytical assumptions of periodicity and meromorphy necessary, which we clarify further in Definition~\ref{defsing}. Denoting \(\AlgJS_k\) as the vector space of singular Jacobi forms of index zero and weight \(k\), Theorem~\ref{prop_strucJS} describes the graded algebra \(\AlgJS=\bigoplus\AlgJS_k\) as the algebra of polynomials \(\CC[\wp,\partial_z\wp, \eisen_4]\), where \(\wp\) is the Weierstra\ss{} function and \(\eisen_4\) is the Eisenstein series of weight 4. Thus, it coincides with the algebra of elliptic forms in the sense of Definition~\ref{dfn_forfonell}. The end of the first section of the article is devoted to determining (proposition~\ref{prop_dimJS}) the dimension of the subspaces \(\AlgJS_k\).

The algebra \(\AlgJS\), like its subalgebra \(\AlgM\), is not stable under the derivation \(\partial_\tau\). This leads to the introduction in Section~\ref{sec_dectiondeux} of the notion of singular quasi-Jacobi forms of index zero, to which are attached by construction a weight \(k\in\N\) and a double depth \((s_1,s_2)\in\N^2\) (see Definition~\ref{dfn_JS}). 
These singular quasi-Jacobi forms are structured into an algebra \(\AlgQJS\) graded by weight and doubly filtered by depth, which Theorem~\ref{thm_strucQJ} describes as the algebra of polynomials in five variables \(\AlgQJS=\CC[\wp,\partial_z\wp, \eisen_4,\eisen_2,\Eisen_1]\), where \(\eisen_2\) is the Eisenstein series of weight \(2\) and depth \((1,0)\), and \(\Eisen_1\) is the first shifted Eisenstein function of weight \(1\) and depth \((0,1)\). The two intermediate subalgebras \({\AlgQJmod}=\CC[\wp,\partial_z\wp, \eisen_4,\eisen_2]\) and \({\AlgQJell}=\CC[\wp,\partial_z\wp, \eisen_4,\Eisen_1]\) between \(\AlgJS\) and \(\AlgQJS\) correspond to quasi-Jacobi forms of depth \((s_1,0)\) and \((0,s_2)\), respectively named quasimodular type and quasielliptic type. 

Section~\ref{sec_trois} of the article is dedicated to constructing formal deformations on each of the four algebras involved and their connections with the classical Rankin-Cohen brackets on the subalgebra \(\AlgM\). The derivation \(\partial_\tau\) of \(\AlgQJS\) being homogeneous of degree \(2\) for the weight, we can introduce in Proposition~\ref{RCtauQ} Rankin-Cohen brackets on \(\AlgQJS\) that constitute a formal deformation of \(\AlgQJS\) (see \cite{zbMATH05156388} and \cite{zbMATH07362171}, following the principle initiated in \cite{MR1280058}). Using the general algebraic arguments of \cite[\texten{Theorem}~6]{zbMATH07362171}, we demonstrate in Theorem~\ref{RCtauA} that the subalgebra \({\AlgQJell}\) is stable under these brackets, which extend those classically defined on \(\AlgM\). The same method allows us to obtain in Theorem~\ref{RCtauE} a formal deformation of the algebra \(\AlgJS\) of elliptic forms extending the Rankin-Cohen brackets on \(\AlgM\) by considering this time bidifferential operators in the derivation \(d=\partial_\tau+\frac14\Eisen_1\partial_z\) (see also with a different proof \cite[\texten{Proposition}~2.15]{zbMATH05953688}). For the case of quasi-Jacobi forms of quasimodular type, it is through a very different strategy based on the notion of transvectants from classical invariant theory (see \cite{zbMATH01516969}) that we obtain in Theorem~\ref{TransB} a formal deformation of the algebra \(\AlgQJmod\).
\section{Singular Jacobi forms}\label{sec_dectiondeux}
\subsection{Elliptic functions associated with a lattice}\label{sub_ellres}
\cite[\texten{Chapter}~V]{zbMATH05500775} 
Let \(\Reseau\) be a lattice in \(\CC\). A meromorphic function \(f\colon\CC\to\CC\) is said to be \emph{elliptic for \(\Reseau\)} if
\[
\forall z\in\CC\enspace\forall w\in \Reseau \quad f(z+w)=f(z). 
\] 
A fundamental example of such a function is the Weierstra\ss{} function associated with the lattice \(\Reseau\) defined by
\[
\forall z\in\CC-\Reseau\quad\wp_{\Reseau}(z)=\frac{1}{z^2}+\sum_{w\in \Reseau-\{0\}}\left(\frac{1}{(z-w)^2}-\frac{1}{w^2}\right). 
\]
Every even elliptic function is a rational function with complex coefficients in \(\wp_\Reseau\) \cite[\texten{Proposition}~V.3.2]{zbMATH05500775}.
For every even integer \(k\geq 4\), we define the complex number
\[
\eisen_{k,\Reseau}=\sum_{w\in \Reseau-\{0\}}w^{-k}. 
\]
The function \(\wp_\Reseau\) satisfies the differential equation 
\begin{equation}\label{eq_derivrho}
\left(\wp_{\Reseau}'\right)^2=\wpdp_\Reseau(\wp_\Reseau) \quad\text{with}\quad \wpdp_\Reseau(X)=4(X^3-15\eisen_{4,\Reseau}X-35\eisen_{6,\Reseau})\in\CC[X]. 
\end{equation}
If \(P_1\), \(Q_1\), \(P_2\), and \(Q_2\) are rational functions, we then have
\[
\left(P_1(\wp_\Reseau)+Q_1(\wp_\Reseau)\wp_{\Reseau}'\right)\left(P_2(\wp_\Reseau)+Q_2(\wp_\Reseau)\wp_{\Reseau}'\right)=\left(P_3(\wp_\Reseau)+Q_3(\wp_\Reseau)\wp_{\Reseau}'\right)
\]
with 
\[
P_3=P_1P_2+\wpdp_\Reseau Q_1Q_{2}\quad\text{and}\quad Q_3=P_1Q_2+Q_1P_2. 
\]

In particular, if \(P\) and \(Q\) are two rational functions in \(\CC(X)\) and if
\[
\widetilde{P}=\frac{P}{P^2-\wpdp_\Reseau Q^2} \quad\text{and}\quad\widetilde{Q}=-\frac{Q}{P^2-\wpdp_\Reseau Q^2}, 
\]
then
\[
\left(P(\wp_\Reseau)+Q(\wp_\Reseau)\wp_{\Reseau}'\right)\left(\widetilde{P}(\wp_\Reseau)+\widetilde{Q}(\wp_\Reseau)\wp_{\Reseau}'\right)=1. 
\]
Thus, the set
\[
\mathcal{E}(\Reseau)=\CC\left(\wp_\Reseau\right)\oplus\CC\left(\wp_\Reseau\right)\wp_{\Reseau}'
\]
is a field. Since \(\CC\left(\wp_\Reseau\right)\) is the field of even elliptic functions, and since if \(f\) is elliptic and odd, then the quotient \(f/\wp_{\Reseau}'\) is elliptic and even, we conclude that the field \(\mathcal{E}(\Reseau)\) is the set of elliptic functions for \(\Reseau\).
\subsection{Elliptic forms}
For all \(\lambda\in\CC^*\), we have
\begin{align}
\wp_{\lambda \Reseau}(z) &= \lambda^{-2}\wp_{\Reseau}\left(\lambda^{-1}z\right) \label{eq_trsfreseauW}\\
\eisen_{k,\lambda \Reseau} &=\lambda^{-k}\eisen_{k,\Reseau}\label{eq_trsfreseaue}
\end{align}
so that we can restrict ourselves to representatives of the equivalence classes of lattices by complex homothety. Any lattice having a basis \((w_1,w_2)\) with \(w_2/w_1\) belonging to the Poincaré half-plane \(\pk\) of complex numbers with strictly positive imaginary part, we restrict to the lattices \(\Z\oplus\tau\Z\) with \(\tau\in\pk\).

We then define, for every even integer \(k\geq 4\), the Eisenstein function of weight \(k\) by
\[
\begin{array}{ccccc}
\eisen_k & \colon & \pk & \to & \CC\\
& & \tau & \mapsto & \eisen_{k,\Z\oplus\tau\Z}. 
\end{array}
\]
This is a modular form of weight \(k\) on \(\sldz\) whose Fourier expansion\footnote{The integer \(\sigma_{k-1}(n)\) is \(\sum_{d\mid n}d^{k-1}\). The sequence \((B_n)_{n\geq 0}\) is defined by the generating series: \[\frac{t}{e^t-1}=\sum_{n=0}^{+\infty}B_n\frac{t^n}{n!}.\]} is given by
\begin{equation}\label{eq_devFEisen}
\eisen_k(\tau)=\frac{2^k\lvert B_{k}\rvert}{k!}\pi^k\left(1-\frac{2k}{B_{k}}\sum_{n=1}^{+\infty}\sigma_{k-1}(n)e^{2i\pi n\tau}\right) \quad(\text{\(k\geq 4\) even}).
\end{equation}
We define \(\eisen_2\) by extending this equality to \(k=2\). The function \(\eisen_2\) is not a modular form.

In a similar manner, we define the Weierstra\ss{} function by
\[
\begin{array}{ccccc}
\wp & \colon & \pkC & \to & \CC\\
& & (\tau,z) & \mapsto & \wp_{\Z\oplus\tau\Z}(z). 
\end{array}
\]
\begin{dfn}\label{dfn_forfonell}
We call an \emph{elliptic form} any element of the ring \(\CC[\wp,\dz\wp,\eisen_4]\) and an \emph{elliptic function} any element of the field of fractions \(\CC(\wp,\dz\wp,\eisen_4)\). 
\end{dfn}

\begin{rem}\label{rem_conventionunedeux}
We will also denote by \(\eisen_k\) the two-variable function induced by the Eisenstein function of weight \(k\), and we shall still call it the Eisenstein function of weight \(k\):
\begin{equation}\label{eq_conventionunedeux}
 \begin{array}{ccccc}
 \eisen_k & \colon &  \pkC & \to & \CC\\
& &  (\tau, z) & \mapsto & \eisen_k(\tau)
\end{array}
\qquad\text{(\(k \geq 2\) even). }
\end{equation}
\end{rem}

\begin{rem}
Our use of the term \emph{elliptic} refers to the theory of elliptic functions associated with a fixed lattice, described using the Weierstra\ss{} function associated with that lattice and its derivative. The intention to ``free'' the lattice requires adding \(\eisen_4\) and \(\eisen_6\). The differential equation
\begin{equation}\label{eq_differfonda}
(\dz\wp)^2 - 4\wp^3 + 60\eisen_4\wp + 140\eisen_6 = 0
\end{equation}
then forces us to remove the function \(\eisen_6\) from the generators: indeed, it can be expressed as a polynomial in the algebraically independent functions \(\wp\), \(\dz\wp\), and \(\eisen_4\) (see Theorem~\ref{prop_strucJS}). We will give another independent proof of~\eqref{eq_differfonda} later in this text (see equation~\eqref{eq_diffwp}).
\end{rem}


\subsection{Singular Jacobi forms of index zero}
\subsubsection{Action of the Jacobi group\texorpdfstring{ on \(\pkC\) et \(\fonc\)}{}}

The multiplicative group \(G=\sldz\) acts on the additive group \(H=\Z^2\) from the right by
\begin{equation*}\label{Gsur H}
\forall g=\begin{pmatrix}a & b\\ c & d\end{pmatrix}\in G, \ \forall \Lambda=(\lambda,\mu)\in H, \ \Lambda g=(\lambda a+\mu c,\lambda b+\mu d)
\end{equation*}
The Jacobi group is the semidirect product \(G\ltimes H=\Jac\) which is derived from this, with the product 
\begin{equation*}
\forall g,g'\in G, \ \forall \Lambda,\Lambda'\in H, \  (g,\Lambda)(g',\Lambda')=(gg', \Lambda g'+\Lambda').
\end{equation*}
The groups \(G\) and \(H\) act on \(\pkC\) from the left as follows:
\begin{align*}
\forall g=\begin{pmatrix}a & b\\ c & d\end{pmatrix}\in G, \ \forall(\tau,z)\in\pkC, \ g(\tau,z)&=\left(\frac{a\tau+b}{c\tau+d},\frac{z}{c\tau+d}\right),\\
\forall \Lambda=(\lambda,\mu)\in H, \ \forall(\tau,z)\in\pkC, \ \Lambda(\tau,z)&=(\tau,z+\lambda\tau+\mu).
\end{align*}
This leads to a right action \(\vert_G\) of \(G=\sldz\) and a right action \(\vert_H\) of \(H=\Z^2\) on the algebra of functions \(\fonc\) defined by
\begin{align}
\forall g=\begin{pmatrix}a & b\\ c & d\end{pmatrix}\in G, \ \forall f\in\fonc,\ &f\vert_G g : (\tau,z)\mapsto f\left(\frac{a\tau+b}{c\tau+d},\frac{z}{c\tau+d}\right)\label{actGsurFonc}\\
\forall \Lambda=(\lambda,\mu)\in H, \ \forall f\in\fonc,\ &f\vert_H\Lambda : (\tau,z)\mapsto f(\tau,z+\lambda\tau+\mu).\label{actHsurFonc}
\end{align}
These two actions are compatible in the sense that
\begin{equation*}
\forall g\in G,\ \forall\Lambda\in H,\ \forall f\in\fonc, \ (f\vert_G g)\vert_H \Lambda g=(f\vert_H \Lambda)\vert_G g.
\end{equation*}
This allows us to deduce a right action of the Jacobi group \(\Jac\) on the algebra of functions \(\fonc\):
\begin{equation}\label{actJsurFonc}
\forall g\in G,\ \forall\Lambda\in H,\ \forall f\in\fonc, \ f\vert_{G\ltimes H}(g,\Lambda)=(f\vert_G g)\vert_H \Lambda.
\end{equation}
In other words, for every \((\tau,z)\in\pkC\),
\begin{equation}
f\vert_{G\ltimes H}\left(\begin{pmatrix}a & b\\ c & d\end{pmatrix},(\lambda,\mu)\right)(\tau,z)=f\left(\frac{a\tau+b}{c\tau+d},\frac{z+\lambda\tau+\mu}{c\tau+d}\right). 
\end{equation}
More generally, if \(\nu\) is a map from \(\Jac\) to \(\fonc\), then the map 
\begin{equation}\label{actJsurFoncnu}
(f,(g,\Lambda))\mapsto\nu(g,\Lambda)\left(f\vert_{G\ltimes H}(g,\Lambda)\right)
\end{equation}
defines a right action of the Jacobi group on the algebra \(\fonc\) if and only if \(\nu\) is a 1-cocycle for the action~\eqref{actJsurFonc}, meaning it satisfies
\begin{equation}\label{def1cocycle}
\nu\left((g,\Lambda)(g',\Lambda')\right)=\left(\nu(g,\Lambda)\vert_{G\ltimes H}(g',\Lambda')\right)\nu(g',\Lambda').
\end{equation}
Such a 1-cocycle \(\nu\) can be obtained from a 1-cocycle \(\nu_G\) for the action~\eqref{actGsurFonc} of \(G\) and a 1-cocycle \(\nu_H\) for the action~\eqref{actHsurFonc} of \(H\) by setting
\begin{equation}\label{cocycleGH}\forall(g,\Lambda)\in G\ltimes H, \ \nu(g,\Lambda)=(\nu_G(g)\vert_H\Lambda)\nu_H(\Lambda)
\end{equation}
which satisfies relation~\eqref{def1cocycle} if and only if we have the compatibility condition
\begin{equation}\label{compatible1cocycle}\forall(g,\Lambda)\in G\ltimes H, \ \left(\nu_G(g)\vert_H\Lambda g\right)\nu_H(\Lambda g)=\nu_G(g)\left(\nu_H(\Lambda)\vert_G g\right).
\end{equation}

Let \(j\colon\sldz\to\fonc\) and \(\ell\colon\sldz\to\fonc\) be defined for \(g=\begin{psmallmatrix}a & b\\c & d\end{psmallmatrix}\) and \((\tau,z)\in\pkC\) by
\[
j(g)(\tau,z)=c\tau+d,\qquad\ell(g)(\tau,z)=\fe\left(-\frac{cz^2}{c\tau+d}\right)
\]
where \(\fe\colon\xi\mapsto\exp(2\ic\pi\xi)\). These are \(1\)-cocycles of \(\sldz\) into \(\fonc\). For all nonnegative integers \(k\) and \(m\), the application \(j^k\ell^m\) is therefore also a \(1\)-cocycle.

Let \(p\colon\Z^2\to\fonc\) be defined for \(\Lambda=(\lambda,\mu)\) and \((\tau,z)\in\pkC\) by
\[
p(\Lambda)(\tau,z)=\fe(\lambda^2\tau+2\lambda z). 
\]
This is a \(1\)-cocycle of \(\Z^2\) into \(\fonc\). For every nonnegative integer \(m'\), the application \(p^{m'}\) is also a \(1\)-cocycle.

Following the construction of~\eqref{cocycleGH}, we then consider the application
\[
\begin{array}{ccccc}
\nu_{k,m,m'}	& \colon	& \sldz\ltimes\Z^2	& \to		& \fonc\\
	& 		& (g,\Lambda) 		& \mapsto	& \left((j^k\ell^m)(g)\vert_{\Z^2}\Lambda\right) p^{m'}(\Lambda).
\end{array}
\]
The compatibility condition~\eqref{compatible1cocycle} is satisfied if and only if \(m'=m\), and we deduce that \(\nu_{k,m,m}=\nu_{k,m}\) is a \(1\)-cocycle for the action~\eqref{actJsurFoncnu} of \(\Jac\). 

Finally, if \(k\) and \(m\) are nonnegative integers, we define an action of \(\Jac\) on \(\fonc\) by
\begin{equation}\label{eq_vertkm}
\left(f\vert_{k,m}A\right)(\tau,z)=(c\tau+d)^{-k}\fe^{m}\left(-\frac{c(z+\lambda\tau+\mu)^2}{c\tau+d}+\lambda^2\tau+2\lambda z\right)f\left(\frac{a\tau+b}{c\tau+d},\frac{z+\lambda\tau+\mu}{c\tau+d}\right)
\end{equation}
for any \(A=(g,\Lambda)=\left(\begin{psmallmatrix}a & b\\c & d\end{psmallmatrix},(\lambda,\mu)\right) \in \Jac\), with \(\fe^{m}(\xi)=\exp(2\ic m\pi\xi)\).
%

\subsubsection{Definition and fundamental examples}\label{subsec_jacosingexple}
\begin{dfn}\label{defsing}
A function \(f\colon\pkC\to\CC\) is \emph{singular} if:
\begin{itemize}
\item for all \(\tau\in\pk\), the function \(z\mapsto f(\tau,z)\) is 1-periodic, meromorphic on \(\CC\), and its only poles are the points of the lattice \(\Z\oplus\tau\Z\), all of the same order, which is independent of \(\tau\);
\item the function \(\tau\mapsto f(\tau,z)\) is 1-periodic;
\item the Laurent coefficients of \(z\mapsto f(\tau,z)\) at \(0\) are holomorphic functions on \(\pk\) and at infinity.
\end{itemize}
We denote by \(\Sing\) the set of singular functions.
\end{dfn}

\begin{rem}
Let us clarify the third condition: let \(A_n\) be the \(n\)-th Laurent coefficient of \(z\mapsto f(\tau,z)\) at \(0\). By the second condition, the functions \(A_n\) are 1-periodic. We therefore require that they be holomorphic on \(\pk\) and have a Fourier expansion of the form
\[
A_n(\tau)=\sum_{r=0}^{+\infty}\widehat{A_n}(r)\fe(r\tau).
\]
\end{rem}

\begin{dfn}\label{eq_dfnJacobiSing}
Let \(k\) and \(m\) be nonnegative integers. A singular function \(f\colon\pkC\to\CC\) is a \emph{singular Jacobi form}\footnote{The definition of meromorphic Jacobi forms does not seem to be established. We draw inspiration from~\cite[\S~3.2]{zbMATH06346312}.} of index \(m\) and weight \(k\) if it satisfies \(f\vert_{k,m}A=f\) for all \(A\in\Jac\). 
\end{dfn}

Explicitly, a singular function is a singular Jacobi form of index \(m\) and weight \(k\) if and only if it satisfies the following two relations:
\begin{itemize}
\item for all \((\lambda,\mu)\in\Z^2\),
\begin{equation}\label{eq_trsfellf}
f(\tau,z+\lambda\tau+\mu)=e^{-2i\pi m(\lambda^2\tau+2\lambda z)}f(\tau,z)\text{ ;}
\end{equation}
\item for all \(\bigl(\begin{smallmatrix}a & b\\ c & d\end{smallmatrix}\bigr)\in\sldz\), 
\begin{equation}\label{eq_trsfmodf}
f\left(\frac{a\tau+b}{c\tau+d},\frac{z}{c\tau+d}\right)=e^{2i\pi mcz^2/(c\tau+d)}(c\tau+d)^kf(\tau,z)\text{ ;}
\end{equation}
\end{itemize}

\begin{rem}
After circulating an initial version of this work, Jan-Willem M. van Ittersum informed us about his article~\cite{zbMATH07634708}, in which he introduces the concept of \emph{strictly meromorphic Jacobi form}. We propose an alternative presentation of a particular case of his and, in particular, we emphasize the analytical context by avoiding the notion of meromorphic form in several variables.
\end{rem}

We fix a matrix \(g=\bigl(\begin{smallmatrix}a & b\\ c & d\end{smallmatrix}\bigr)\) in \(\sldz\) and \((\lambda, \mu) \in \Z^2\). We have
\[
\wp\left(\frac{a\tau+b}{c\tau+d},\frac{z}{c\tau+d}\right)=\wp_{\Z\oplus\frac{a\tau+b}{c\tau+d}\Z}\left(\frac{z}{c\tau+d}\right).
\]
Now, \(\Z\oplus\frac{a\tau+b}{c\tau+d}\Z=\frac{1}{c\tau+d}\left(\Z\oplus\tau\Z\right)\), the equality~\eqref{eq_trsfreseauW} then implies
\[
\wp_{\Z\oplus\frac{a\tau+b}{c\tau+d}\Z}\left(\frac{z}{c\tau+d}\right)=(c\tau+d)^2\wp_{\Z\oplus\tau\Z}(z)
\]
that is to say
\begin{equation}\label{eq_trsfmodwp}
\wp\left(\frac{a\tau+b}{c\tau+d},\frac{z}{c\tau+d}\right)=(c\tau+d)^2\wp(\tau,z).
\end{equation}
On the other hand, by the definition of elliptic functions
\begin{equation}\label{eq_trsfellwp}
\wp\left(\tau,z+\lambda\tau+\mu\right)=\wp(\tau,z).
\end{equation}
Let us denote \(\dz=\partial/\partial z\). By differentiating \eqref{eq_trsfmodwp} and \eqref{eq_trsfellwp}, we find
\begin{equation}\label{eq_trsfmoddzwp}
\left(\dz\wp\right)\left(\frac{a\tau+b}{c\tau+d},\frac{z}{c\tau+d}\right)=(c\tau+d)^3\dz\wp(\tau,z).
\end{equation}
and
\begin{equation}\label{eq_trsfelldzwp}
\left(\dz\wp\right)\left(\tau,z+\lambda\tau+\mu\right)=\dz\wp(\tau,z).
\end{equation}
With the convention described in Remark~\ref{rem_conventionunedeux}, the function \(\eisen_4\) satisfies the equation
\begin{equation}\label{ploufplouf}
\eisen_4\left(\frac{a\tau+b}{c\tau+d}, \frac{z}{c\tau+d}\right)=(c\tau+d)^4\eisen_4(\tau, z).
\end{equation}
The Laurent expansion of \(\wp\) is given by
\begin{equation}\label{eq_delLaurent}
\wp(\tau,z)=\frac{1}{z^2}+\sum_{n=1}^{+\infty}(2n+1)\eisen_{2n+2}(\tau)z^{2n}
\end{equation}
\cite[\texten{Proposition}~V.2.5]{zbMATH05500775}, which shows that \(\wp\) (and thus \(\dz\wp\)) are singular.

The relations \eqref{eq_trsfmodwp} to \eqref{eq_delLaurent} thus show that \(\wp, \dz\wp\) and \(\eisen_4\) are singular Jacobi forms of index zero and weights \(2, 3\), and \(4\), respectively. The remainder of the section aims 
to clarify the analytical framework of Proposition 2.8 of~\cite{zbMATH05953688}. This is a new response to a similar objective pursued by~\cite[Proposition 2.7]{zbMATH07634708}.

\begin{thm}\label{prop_strucJS}

\begin{enumerate}
\item The functions \(\wp, \dz\wp\) and \(\eisen_4\) are algebraically independent.
\item The algebra of elliptic forms is graded by weight. We denote \(\displaystyle\AlgJS=\CC[\wp,\dz\wp,\eisen_4]=\bigoplus_{k \in \N} \AlgJS_k\) where \(\AlgJS_k\) is the set of elements \(\displaystyle\sum_{\substack{(a,b,c)\in\N^3\\ 2a+3b+4c=k}}\alpha(a, b, c)\wp^a\left(\dz\wp\right)^b\eisen_4^c\) with \(\alpha(a, b, c) \in \CC\).
\item For all \(k \geq 0\), \(\AlgJS_k\) is the set of singular Jacobi forms of index zero and weight \(k\).
\end{enumerate}
\end{thm}

\begin{proof}
Let us first show the algebraic independence of \(\wp\), \(\dz\wp\), and \(\eisen_4\). For all \(\tau \in \pk\), we have \(\dz\wp\left(\tau, \dfrac{\tau}2\right)=0\) according to \cite[\texten{Lemma}~V.2.8]{zbMATH05500775}. Thanks to~\eqref{eq_differfonda}, there is an algebraic dependence relation among the functions \(\eisen_4\), \(\eisen_6\), and \(\widetilde{\wp}\colon\tau\mapsto\wp(\tau,\tau/2)\):
\[
\widetilde{\wp}^3-15\eisen_4\widetilde{\wp}-35\eisen_6=0.
\]
Since \(\eisen_4\) and \(\eisen_6\) are algebraically independent, we conclude that the functions \(\eisen_4\) and \(\widetilde{\wp}\) are also algebraically independent. Assume now that \(\wp\), \(\dz\wp\), and \(\eisen_4\) are algebraically dependent. There would exist an integer \(N\geq 1\) and a non-zero sequence of complex numbers \(\alpha_{k,\ell}^{(i)}\) such that 
\[
\sum_{i=0}^Nf_i(\dz\wp)^i=0\quad\text{with}\quad f_i=\sum_{k,\ell}\alpha_{k,\ell}^{(i)}\eisen_4^\ell\wp^k. 
\]
By specializing this equality at \(z=\tau/2\), we show that \(\tau \mapsto f_0\left(\tau, \dfrac{\tau}2\right)\) is zero, and by induction, all \(\tau \mapsto f_i\left(\tau, \dfrac{\tau}2\right)\) are zero. By the algebraic independence of \(\eisen_4\) and \(\widetilde{\wp}\), it follows that all the \(\alpha_{k,\ell}^{(i)}\) are zero, leading to a contradiction. This proves point (1). Point (2) follows from Definition~\ref{dfn_forfonell}.

Now let us prove (3). By applying~\eqref{eq_trsfmodf} to \((a,b,c,d)=(-1,0,0,-1)\), any singular Jacobi form of index zero and weight \(k\) is even in the variable \(z\) if \(k\) is even and odd in the variable \(z\) if \(k\) is odd. 

Let \(f\) be a singular Jacobi form of index zero and weight \(k\). For all \(\tau\), the function \(z\mapsto f(\tau,z)\) is an elliptic function associated with the lattice \(\Z\oplus\tau\Z\), whose poles are points of the lattice. In \(\qd{\CC}{\Z\oplus\tau\Z}\), this function therefore has at most one pole (which can be multiple), and that is at \(0\). 

\paragraph{Case of even weight}
If \(k\) is even, for all \(\tau\) there exists \(P_\tau\in\CC[X]\) such that 
\[
f(\tau,z)=P_\tau\left(\wp(\tau,z)\right) 
\] 
and the degree \(n_0\) of \(P_{\tau}\) is half the order of the pole of \(z\mapsto f(\tau,z)\) at \(0\) \cite[\texten{Proposition}~V.3.1]{zbMATH05500775}. It is therefore independent of \(\tau\), and there exist functions \(a_0,\dotsc,a_{n_0}\) of \(\pk\) into \(\CC\) such that
\begin{equation}\label{eq_poljacpair}
f(\tau,z)=\sum_{j=0}^{n_0}a_j(\tau)\wp(\tau,z)^j.
\end{equation}
Considering~\eqref{eq_trsfmodf} and \eqref{eq_trsfmodwp}, we have
\[
\sum_{j=0}^{n_0}(c\tau+d)^{2j}a_j\left(\frac{a\tau+b}{c\tau+d}\right)\wp(\tau,z)^j=\sum_{j=0}^{n_0}(c\tau+d)^{k}a_j(\tau)\wp(\tau,z)^j.
\]
The family \((\wp_{{\Z\oplus\tau\Z}}^j)_{j\in\mathbb{N}}\) is linearly independent. We deduce that each \(a_j\) is a weakly modular function\footnote{In the sense of~\cite{MR0498338}, that is, meromorphic on the Poincaré half-plane and satisfying modularity relations;.} of weight \(k-2j\).

Let us show that the \(a_j\) are holomorphic on \(\pk\) and at infinity. The equality between the Laurent expansion
\[
\sum_{n=0}^{+\infty}A_n(\tau)z^{2n-2n_0}
\]
of \(z \mapsto f(\tau, z)\) at \(0\) and the equality~\eqref{eq_poljacpair} leads, thanks to~\eqref{eq_delLaurent}, to 
\[
\sum_{n=0}^{+\infty}A_n(\tau)z^{2n}=\sum_{j=0}^{n_0}a_j(\tau)\left(\sum_{r=0}^{+\infty}\epsilon_r(\tau)z^{2r}\right)^jz^{2n_0-2j}
\]
where the holomorphic functions \(\epsilon_r\) on \(\pk\) are defined by \(\epsilon_0=1\), \(\epsilon_1=0\), and \(\epsilon_r=(2r-1)\eisen_{2r}\) if~\(r\geq 2\). We deduce
\[
A_r=a_{n_0-r}+\sum_{j=n_0-r+1}^{n_0}a_j\sum_{\alpha_1+\dotsm+\alpha_j=r+j-n_0}\epsilon_{\alpha_1}\dotsm\epsilon_{\alpha_j}.
\]
By induction, we obtain that the functions \(a_j\) are holomorphic on \(\pk\) and at infinity.

Finally, the functions \(a_j\) are modular forms, hence elements of \(\CC[\eisen_4,\eisen_6]\). Thus, a singular Jacobi form of index zero and even weight is an element of \(\CC[\eisen_4,\eisen_6,\wp]\subset\CC[\eisen_4,\wp,(\dz\wp)^2]\subset\CC[\eisen_4,\wp,\dz\wp]\).

\paragraph{Case of odd weight} If \(k\) is odd, then \(f\dz\wp\) is a singular Jacobi form of index zero and even weight~\(k+3\). We conclude that \(f\dz\wp\in\CC[\eisen_4,\wp,(\dz\wp)^2]\) and that there exist polynomials \(P\) and \(Q\) such that
\[
f=\frac{1}{\dz\wp}P(\eisen_4,\wp)+Q(\eisen_4,\wp,(\dz\wp)^2)\dz\wp.
\]
For all \(\tau \in \pk\), the function \(z \mapsto f(\tau, z)\) does not have a pole at \(z=\tau/2\), hence \[P\left(\eisen_4,\widetilde{\wp}\right)=0.\] By the algebraic independence of \(\eisen_4\) and \(\widetilde{\wp}\), the polynomial \(P\) must be zero. Thus \(f \in \CC[\wp,\dz\wp,\eisen_4]\).
\end{proof}
\subsubsection{Dimension ot the space \texorpdfstring{\(\AlgJS_k\)}{of singular Jacobi forms}}
For any integer \(k \geq 0\), a basis of the space \(\AlgJS_k\) is 
\begin{equation}\label{eq_base}
\left\{\wp^a(\dz\wp)^b\eisen_4^c \colon (a,b,c)\in\N^3, 2a+3b+4c=k\right\}. 
\end{equation}
The equation \(2a+3b+4c=k\) is equivalent to \(4a+6b=2k-8c\), and since the algebra \(\CC[\eisen_4, \eisen_6]\) of modular forms for \(\sldz\) is generated by a function of weight \(4\) and one of weight \(6\), we deduce that
\begin{equation}\label{eq_frommodular}
\dim\JS{k}=\sum_{c=0}^{\lfloor k/4 \rfloor}\fd(2k-8c)
\end{equation}
where for all \(j\in\N\), \(\fd(j)\) denotes the dimension of the space of modular forms of weight \(j\), explicitly given by
\begin{equation}\label{eq_d}
\fd(j)=\left\lfloor\frac{j}{12}\right\rfloor+\begin{cases*}
0 & if \(12\) divides \(j-2\) \\
1 & otherwise. 
\end{cases*}
\end{equation}
Although there are no modular forms of negative weights, and \(\fd(j)\) should be zero for \(j<0\), we adopt a different convention to proceed with the following calculations, focusing not on the modular aspect of \(\fd\) but rather on its combinatorial aspect. We extend the definition of \(\fd\) by~\eqref{eq_d} to all integers \(\Z\). Then, we have \(\fd(j+12)=\fd(j)+1\) for all \(j\in\Z\). 

Let \(x\) be a real number, and let \(\projZ{x}\) denote the nearest integer to \(x\) (with the convention \(\projZ{n+1/2}=n\) for all \(n \in \Z\)).
\begin{prop}\label{prop_dimJS}
For any natural number \(k\), the dimension \(\fds(k)\) of the space of singular Jacobi forms of index zero and weight \(k\) is given by 
\begin{equation}\label{eq_dimexplicit}
\fds(k)=\dim\JS{k}=
\projZ*{\frac{\left(k+3\eta(k)\right)^2}{48}}\quad\text{with}\quad\eta(k)=\begin{dcases*} 1 & if \(k\) is odd\\ 2 & otherwise. \end{dcases*}
\end{equation}
The generating series of these dimensions is
\[
\sum_{k\in\N}\fds(k)\cdot z^k=\frac{1}{(1-z^2)(1-z^3)(1-z^4)}
\]
and we have the recurrence relations: 
\[
\fds(2k+3)=\fds(2k)\quad\text{and}\quad\fds(2k+13)=\fds(2k+1)+k+5
\]
for all integers \(k\). The first values are given by
\[
\begin{array}{|c||c|c|c|c|c|c|c|c|}
\hline
k & 0 & 1 & 2  & 4 & 6 & 8 & 10 & 12\\
\hline
\fds(k) & 1 & 0 & 1 & 2 & 3 & 4 & 5 & 7\\
\hline
\end{array}
\]

\end{prop}
\begin{proof}
By counting the elements of the basis~\eqref{eq_base} of \(\JS{k}\), we find that the generating series of \(\fds\) is
\[
\sum_{k\in\N}\#\left\{(a,b,c)\in\N^3\colon k=2a+3b+4c\right\}z^k=\sum_{a\in\N}z^{2a}\sum_{b\in\N}z^{3b}\sum_{c\in\N}z^{4c}=\frac{1}{(1-z^2)(1-z^3)(1-z^4)}.
\]
We then deduce
\[
\sum_{k\in\N}\fds(2k)z^k=\frac{1}{(1-z)^3(1+z)(1+z+z^2)}
\]
and
\[
\sum_{k\in\N}\fds(2k+1)z^k=\frac{z}{(1-z)^3(1+z)(1+z+z^2)}.
\]
This immediately gives us
\begin{equation}\label{eq_pairimpair}
\fds(2k)=\fds(2k+3)
\end{equation}
for all integers \(k\). Considering~\eqref{eq_frommodular} and the extension to \(\Z\) of~\eqref{eq_d}, we get
\[
\fds(2k+13)=\fds(2k+1)+\sum_{c=\left\lfloor\frac{2k+1}{4}\right\rfloor+1}^{\left\lfloor\frac{2k+1}{4}\right\rfloor+3}\fd(4k-8c+2)+2\left\lfloor\frac{2k+1}{4}\right\rfloor+8. 
\]
From this, we deduce the second recurrence relation:
\begin{equation}\label{eq_impairimpair}
\fds(2k+13)=\fds(2k+1)+k+5
\end{equation}
for all integers \(k\). 

The function \(\varphi : k\mapsto\frac{\left(k+3\eta(k)\right)^2}{48}\) also satisfies the relations~\eqref{eq_pairimpair} and~\eqref{eq_impairimpair}, and therefore so does \(\projZ{\varphi}\). We conclude that
\(\displaystyle\fds(k)=
\projZ*{\frac{\left(k+3\eta(k)\right)^2}{48}}\)
for all integers \(k\), by comparing the values for \(k\in\{0,1,2,4,6,8,10,12\}\).
\end{proof}
\begin{rem}
From the generating series of \(\left(\fds(k)\right)_{k\in\N}\), we deduce that this sequence is \(\left(t(k+3)\right)_{k\in\N}\), where \(t\) is the Alcuin sequence~\cite{AlcuinSeq}. The explicit formula is then proven in~\cite[\texten{Theorem}~1]{zbMATH06071095}. The equations~\eqref{eq_pairimpair} and~\eqref{eq_impairimpair} are given in this context in~\cite{zbMATH03646988} and proven in~\cite{zbMATH03654914}.
\end{rem}
\begin{rem}
We can systematically obtain similar formulas for the dimensions of the spaces considered in this text. A discussion on these formulas is provided in appendix.
\end{rem}
\subsubsection{Application to the differential equation of the Weierstra\ss{} function}
The modular form \(\eisen_6\) is a singular Jacobi form of index zero and weight \(6\). The dimension of \(\JS{6}\) is \(3\), with a basis being \(\left((\dz\wp)^2,\wp^3,\wp\eisen_4\right)\). Thus, \(\eisen_6\) is a linear combination of these three functions. By identifying the terms in \(z^{-6}\), \(z^{-2}\), and \(z^0\) in the Laurent expansion at \(z=0\), we obtain:
\begin{equation}\label{eq_diffwp}
\eisen_6=-\frac{1}{140}(\dz\wp)^2+\frac{1}{35}\wp^3-\frac{3}{7}\wp\eisen_4. 
\end{equation}
Thus, we recover the differential equation of the Weierstra\ss{} \(\wp\) function, which is central in the theory of elliptic curves \cite[\texten{Theorem}~V.3.4]{zbMATH05500775}.

\section{Singular quasi-Jacobi forms of index zero}\label{sec_trois}
\subsection{Action and differentiation}
The action of \(\Jac\) on \(\pkC\) is given by the map \(\AC\):
\begin{equation}\label{eq_defAC}
\begin{array}{ccccc}
\AC	& \colon	& \Jac	& \to		&	\left(\pkC\right)^{\pkC}\\
	&		& A=(g, \Lambda)=\left(\begin{pmatrix}a & b\\ c & d\end{pmatrix},(\lambda,\mu)\right)		& \mapsto	&	\begin{array}{ccc}
									\pkC		& \to		& \pkC\\
									(\tau,z)	& \mapsto	& A\cdot(\tau,z)=\left(\frac{a\tau+b}{c\tau+d},\frac{z+\lambda\tau+\mu}{c\tau+d}\right). 
								\end{array}	 
\end{array}								
\end{equation}
By the definition of an action, we have
\begin{equation}\label{eq_Acac}
\AC(AB)=\AC(A)\circ\AC(B).
\end{equation}
We calculate
\begin{equation}\label{eq_derivH}
\frac{\partial\AC}{\partial\tau}=\left(\frac{1}{\J^2},-\frac{\Y}{\J}\right)\quad\text{and}\quad\frac{\partial\AC}{\partial z}=\left(0,\frac{1}{\J}\right)
\end{equation}
with
\[
\begin{array}{ccccc}
\J	& \colon	& \Jac	& \to		&	\fonc\\
	&		& \left(\begin{psmallmatrix}a & b\\c & d\end{psmallmatrix},(\lambda,\mu)\right)		& \mapsto	&	\begin{array}{ccc}
									\pkC		& \to		& \CC\\
									(\tau,z)	& \mapsto	& c\tau+d
								\end{array}	 
\end{array}								
\]
and
\[
\begin{array}{ccccc}
\Y	& \colon	& \Jac	& \to		&	\fonc\\
	&		& \left(\begin{psmallmatrix}a & b\\c & d\end{psmallmatrix},(\lambda,\mu)\right)		& \mapsto	&	\begin{array}{ccc}
									\pkC		& \to		& \CC\\
									(\tau,z)	& \mapsto	& \dfrac{cz+c\mu -d\lambda}{c\tau+d}. 
								\end{array}	 
\end{array}								
\]
By defining
\[
\begin{array}{ccccc}
\X	& \colon	& \Jac	& \to		&	\fonc\\
	&		& \left(\begin{psmallmatrix}a & b\\c & d\end{psmallmatrix},(\lambda,\mu)\right)		& \mapsto	&	\begin{array}{ccc}
									\pkC		& \to		& \CC\\
									(\tau,z)	& \mapsto	& \dfrac{c}{c\tau+d}
								\end{array}	 
\end{array}								
\]
we have 
\begin{subequations}\label{eq_derivees}
\begin{alignat}{5}
\frac{\partial\J}{\partial\tau}	& =\X\J	&\qquad\qquad& &\frac{\partial\Y}{\partial\tau}	& =-\X\Y	&\qquad\qquad& & \frac{\partial\X}{\partial\tau}	&=-\X^2\label{eq_deriveestau}\\
\frac{\partial\J}{\partial z}		&=0			&& &\frac{\partial\Y}{\partial z}	&=\X	&& &\frac{\partial\X}{\partial z}	&=0.\label{eq_deriveesz}
\end{alignat}
\end{subequations}
It is clear that the functions \(\J\), \(X\), and \(\Y\) are algebraically independent over \(\CC\).

It follows from \eqref{eq_derivees} that the algebra \(\CC[\J,\X,\Y]\) is stable under the differentiation with respect to \(\tau\) and \(z\).
The proof of the following proposition shows that the notion of a cocycle allows us to understand the derivatives of the action with respect to \(z\) and \(\tau\).
\begin{prop}\label{prop_JXYcocycles}

We have, for the functions \(\J\), \(\X\), and \(\Y\) and the action defined in \eqref{eq_vertkm}, the following 1-cocycle relations: \(\forall(A, B)\in\left(\sldz\ltimes\Z^2\right)^2 \)
\[ 
\J(AB)=\left(\J(A)\vert_{0,0}B\right)\J(B), \quad \Y(AB)=\Y(A)\vert_{1,0}B+\Y(B),\quad \X(AB)=\X(A)\vert_{2,0}B+\X(B).
\]
\end{prop}
\begin{proof}
The first relation on \(\J\) is well known and easy to verify. 
For the second formula, we differentiate~\eqref{eq_Acac} with respect to \(\tau\). Denoting \(\AC=(\AC_1,\AC_2)\), we find:
\[
\frac{1}{\J(AB)}\left(\frac{1}{\J(AB)},-\Y(AB)\right)=\frac{\partial\AC(A)}{\partial\tau}\left(\AC(B)\right)\frac{\partial\AC_1(B)}{\partial\tau}+\frac{\partial\AC(A)}{\partial z}\left(\AC(B)\right)\frac{\partial\AC_2(B)}{\partial\tau}
\]
which, using~\eqref{eq_derivH}, leads to
\[
\left(\frac{1}{\J(AB)(x)^2},-\frac{\Y(AB)(x)}{\J(AB)(x)}\right)=\left(\frac{1}{\J(A)(Bx)^2\J(B)(x)^2},-\frac{\Y(A)(Bx)}{\J(A)(Bx)\J(B)(x)^2}-\frac{\Y(B)(x)}{\J(A)(Bx)\J(B)(x)}\right)
\]
where we have denoted \(x=(\tau,z)\). Comparing the second coordinates and using the previous formula, we obtain 
\( \Y(AB)(x)=\J(B)(x)^{-1}\Y(A)(Bx)+\Y(B)(x)\)
which proves the desired relation.

Next, differentiating the cocycle relation of \(\Y\) with respect to \(z\), we find
\[
\frac{\partial\Y(AB)}{\partial z}(x)=\frac{1}{\J(B)(x)}\frac{\partial\Y(A)}{\partial\tau}(Bx)\frac{\partial\AC_1(B)}{\partial z}(x)+\frac{1}{\J(B)(x)}\frac{\partial\Y(A)}{\partial z}(Bx)\frac{\partial\AC_2(B)}{\partial z}(x)+\frac{\partial\Y(B)}{\partial z}(x).
\]
Thanks to~\eqref{eq_deriveesz} and~\eqref{eq_derivH}, we deduce
\[
\X(AB)(x)=\J(B)(x)^{-2}\X(A)(Bx)+\X(B)(x). 
\]
This is the cocycle relation of \(\X\). 
\end{proof}

\subsection{Definition}
\begin{dfn}\label{dfn_JS}
A singular function \(f\colon\pkC\to\CC\) is called a \emph{quasi-Jacobi singular form} (of index zero), of weight \(k\in\N\) and of depth \((s_1,s_2)\in\N^2\) if there exist \(\left(f_{j_1,j_2}\right)_{\substack{0\leq j_1\leq s_1\\ 0\leq j_2\leq s_2}}\in\Sing^{(s_1+1)(s_2+1)}\) such that
\begin{equation}\label{eq_trsfJac}
\forall A\in\Jac\qquad f\vert_{k,0}A=\sum_{j_1=0}^{s_1}\sum_{j_2=0}^{s_2}f_{j_1,j_2}\X(A)^{j_1}\Y(A)^{j_2}.
\end{equation}
where \(f_{s_1, s_2}\) is not identically zero. From now on, we agree to denote \(f\vert_kA :=f\vert_{k, 0}A\), and we will only consider forms of index zero.
It follows from the algebraic independence of \(\X\) and \(\Y\) over \(\CC\) that the decomposition \eqref{eq_trsfJac} is unique. We then define \(\cQ_{j_1,j_2}(f)=f_{j_1,j_2}\), and we call \(s_1\) the \emph{modular depth} of \(f\) and \(s_2\) its \emph{elliptic depth}. The vector space of quasi-Jacobi singular forms of weight \(k\) and depths less than or equal to \(s_1\) and \(s_2\) is denoted by \(\QJS{k}{s_1}{s_2}\); the vector space of quasi-Jacobi singular forms of weight \(k\) is denoted by \(\QJSpoids{k}\).
\end{dfn}

\begin{rem}
The notion introduced here is consistent with that of~\cite[\S2.4]{zbMATH07634708}. While that article presents an approach \emph{via} the notion of \emph{almost Jacobi form}, we favor a direct approach.
\end{rem}

\begin{rem}\label{jacobisinguliereprofnulle}\leavevmode
The choice \(A=\left(\begin{psmallmatrix}1 & 0\\0 & 1\end{psmallmatrix},(0,0)\right)\) implies that \(\cQ_{0,0}(f)=f\). This particularly implies that \(\QJS{k}{0}{0}\) is the space \(\JS{k}\) of Jacobi singular forms of index zero and weight \(k\), as previously encountered.
\end{rem}
\begin{rem}
\begin{itemize}
 \item 
Let \(f \in \QJSpoids{k}\) and \(g \in \QJSpoids{\ell}\), then we have \(fg \in \QJSpoids{k+\ell}\) and 
\[
\cQ_{i,j}(fg)=\sum_{\substack{(\alpha,\beta,\gamma,\delta)\\\alpha+\beta=i\\\gamma+\delta=j}}\cQ_{\alpha,\gamma}(f)\cQ_{\beta,\delta}(g).
\]
\item It follows from the algebraic independence of \(\X, \Y\), and \(\J\) over \(\CC\) that the spaces \(\QJSpoids{k}\) are in direct sum. We can therefore consider the algebra graded by the weight \(\displaystyle\AlgQJS=\bigoplus_{k \in \N}\QJSpoids{k}\), which we will agree to call the algebra of quasi-Jacobi singular forms.
\end{itemize}
\end{rem}

\subsection{Stability under differentiation}
The derivation with respect to \(z\) is zero on the algebra \(\AlgM\) of modular forms. However, \(\AlgM\) is not stable under differentiation with respect to \(\tau\), which justifies the introduction of the algebra \(\AlgQM\) of quasimodular forms\cite{zbMATH06128504,zbMATH05050117}.

The algebra \(\AlgJS\) of singular Jacobi forms is stable under differentiation with respect to \(z\) but is not stable under differentiation with respect to \(\tau\) (as will be seen later, see Remark~\vref{rem_profdz}, \eqref{eq_ddeuxzrho} and \eqref{eq_dtaudzrho}). Here, we show that the algebra \(\AlgQJS\) is stable under each of these derivations. 
\begin{lem}\label{lem_derivAC}
Let \(f\colon\pkC\to\CC\) be differentiable with respect to each variable, then
\begin{equation}\label{eq_derivACz}
\frac{\partial\left(f\vert_kA\right)}{\partial z}=\left.\left(\frac{\partial f}{\partial z}\right)\right\rvert_{k+1} A
\end{equation}
and
\begin{equation}\label{eq_derivACtau}
\frac{\partial\left(f\vert_kA\right)}{\partial\tau}=-k\left(f\vert_kA\right)\X(A)+\left.\left(\frac{\partial f}{\partial\tau}\right)\right\rvert_{k+2} A-\Y(A)\left.\left(\frac{\partial f}{\partial z}\right)\right\rvert_{k+1} A.
\end{equation}
\end{lem}
\begin{proof}
The result is obtained by differentiating with respect to \(z\) and \(\tau\) the definition \(f\vert_kA=\J(A)^{-k}f\left(\AC(A)\right)\), then using~\eqref{eq_derivH} and~\eqref{eq_derivees}.
\end{proof}
\begin{prop}\label{prop_stabder}
The algebra \(\AlgQJS\) is stable under differentiation with respect to \(z\) and \(\tau\). The derivation \(\partial/\partial z\) maps \(\QJS{k}{s_1}{s_2}\) into \(\QJS{k+1}{s_1+1}{s_2}\); the derivation \(\partial/\partial\tau\) maps \(\QJS{k}{s_1}{s_2}\) into \(\QJS{k+2}{s_1+1}{s_2+1}\). Furthermore, for \(f \in \QJSpoids{k}\),
\[
\cQ_{j_1,j_2}\left(\frac{\partial f}{\partial z}\right)=\frac{\partial\cQ_{j_1,j_2}(f)}{\partial z}+(j_2+1)\cQ_{j_1-1,j_2+1}(f)
\]
and
\[
\cQ_{j_1,j_2}\left(\frac{\partial f}{\partial\tau}\right)=\frac{\partial\cQ_{j_1,j_2}(f)}{\partial\tau}+\frac{\partial\cQ_{j_1,j_2-1}(f)}{\partial z}+(k-j_1+1)\cQ_{j_1-1,j_2}(f).
\]
More precisely,
\[
\frac{\partial}{\partial z}\QJS{k}{s_1}{s_2}\subseteq\QJS{k+1}{s_1+1}{s_2-1}+\QJS{k+1}{s_1}{s_2}
\]
and
\[
\frac{\partial}{\partial\tau}\QJS{k}{s_1}{s_2}\subseteq\QJS{k+2}{s_1+1}{s_2}+\QJS{k+2}{s_1}{s_2+1}.
\]
\end{prop}
\begin{proof}
Thanks to~\eqref{eq_derivACz} and Definition~\ref{dfn_JS}, we find
\[
\left.\left(\frac{\partial f}{\partial z}\right)\right\rvert_{k+1} A=\sum_{j_1=0}^{s_1}\sum_{j_2=0}^{s_2}\left(\frac{\partial f_{j_1,j_2}}{\partial z}\X(A)^{j_1}\Y(A)^{j_2}+j_2f_{j_1,j_2}\X(A)^{j_1+1}\Y(A)^{j_2-1}\right).
\]
From this, we deduce the results related to \(\partial/\partial z\).

Moreover, thanks to~\eqref{eq_derivACtau} and Definition~\ref{dfn_JS}, we find
\begin{multline*}
-k\left(f\vert_kA\right)\X(A)+\left.\left(\frac{\partial f}{\partial\tau}\right)\right\rvert_{k+2} A-\Y(A)\left.\left(\frac{\partial f}{\partial z}\right)\right\rvert_{k+1} A=\\
\sum_{j_1=0}^{s_1}\sum_{j_2=0}^{s_2}\left(
\frac{\partial f_{j_1,j_2}}{\partial\tau}\X(A)^{j_1}\Y(A)^{j_2}-j_1f_{j_1,j_2}\X(A)^{j_1+1}\Y(A)^{j_2}-j_2f_{j_1,j_2}\X(A)^{j_1+1}\Y(A)^{j_2}
\right).
\end{multline*}
Using the results related to \(\partial/\partial z\), we then find those related to \(\partial/\partial\tau\).

If \(f\in\QJS{k}{s_1}{s_2}\), then \(\frac{\partial f}{\partial z}\in\QJS{k+1}{s_1+1}{s_2}\), but \(\cQ_{s_1+1,s_2}(\partial f/\partial z)=0\), so \(\frac{\partial f}{\partial z}\in\QJS{k+1}{s_1+1}{s_2-1}+\QJS{k+1}{s_1}{s_2}\). The inclusion for \(\partial/\partial\tau\) is proved in the same way.
\end{proof}
\begin{rem}\label{rem_profdz}
Thus, if \(s_2=0\), then \(\frac{\partial f}{\partial z}\in\QJS{k+1}{s_1}{0}\). In particular, if \(f \in \AlgJS_k\), then \(\dfrac{\partial f}{\partial z} \in \AlgJS_{k+1}\).
\end{rem}
\subsection{Fundamental examples}\label{subsec_qjacfonda} 
The results of this section are summarized in Table~\vref{tab_explefonda}.
\subsubsection{Quasimodular forms}\label{subsub_fondafqmod}
As mentioned in Paragraph~\ref{subsec_jacosingexple}, we identify from now on any function \(f : \pk \to \CC\) with the function \(f\colon\pkC\to\CC\) defined by \(f(\tau,z)=f(\tau)\). Through this identification, any modular form of weight \(k\) is a singular quasi-Jacobi form of weight \(k\) and depth \((0,0)\). The \(n\)-th derivative (with respect to \(\tau\)) of a modular form of weight \(k\) is then a singular quasi-Jacobi form of weight \(k+2n\) and depth \((n,0)\). Similarly, \(\eisen_2\) is a singular quasi-Jacobi form of weight 2 and depth \((1,0)\) with \(\cQ_{1,0}\left(\eisen_2\right)=-2\ic\pi\). Since the algebra of quasimodular forms is generated by the modular forms \(\eisen_4\) and \(\eisen_6\) and by the quasimodular form \(\eisen_2\), we have thus shown that all quasimodular forms are singular quasi-Jacobi forms.

\subsubsection{The first shifted Eisenstein function}
The shifted Eisenstein series of weight 1 is the series defined on \(\pkC\) by
\[
\Eisen_1(\tau,z)=\lim_{M\to+\infty}\sum_{m=-M}^M\left(\lim_{N\to+\infty}\sum_{\substack{n=-N\\ m=0\Rightarrow n\neq 0}}^N\frac{1}{z+m+n\tau}\right)
\]
\cite[\texten{Chapter}~III, \S 2]{zbMATH01236956}. This function is well-defined and admits a Laurent series expansion
\begin{equation}\label{eq_LaurentEun}
\Eisen_1(\tau,z)=\frac{1}{z}-\sum_{n=0}^{+\infty}\eisen_{2n+2}(\tau)z^{2n+1}
\end{equation}
with the series converging on any punctured open disk centered at \(z=0\) with a radius less than~\(\abs{\tau}\) (see \cite[\texten{Chapter}~III, eq. (9)]{zbMATH01236956}). It satisfies the equation:
\[
\forall A\in\Jac\qquad\Eisen_1\vert_1A=\Eisen_1+2\ic\pi\Y(A)
\]
\cite[\texten{Lemma}~1]{hal03132764}\footnote{In this work, \(\mathrm{J}_1\) was used to denote what we refer to here as \(\frac{1}{2i\pi}\Eisen_1\).}; the function \(z\mapsto\Eisen_1(\tau,z)\) is meromorphic, with its poles located at the lattice points \(\Z+\tau\Z\), and they are simple. Thus, the function \(\Eisen_1\) is a singular quasi-Jacobi form of weight 1 and depth \((0,1)\).

\begin{rem}
We use Weil's convention from~\cite{zbMATH07634708}, reserving uppercase letters for functions intrinsically depending on two variables and lowercase letters for functions intrinsically depending on one variable. The function \(\Eisen_1\) and the generalizations obtained by replacing \(z+m+n\tau\) by \((z+m+n\tau)^k\) are particular cases of what Charollois \& Sczech~\cite{zbMATH06696479} call Kronecker–Eisenstein series. Our function \(\Eisen_1\) is their function \(K^*(z,0,1,\tau)=Ser(0,0,z,\tau)\).
\end{rem}

\begin{table}
\begin{center}
\begin{tabular}{|>{$}c<{$}|>{$}c<{$}|>{$}c<{$}|>{$}c<{$}|}
\hline
\text{Function} & \text{Weight} & \begin{tabular}{c}Depth \\ \((s_1,s_2)\)\end{tabular} & \cQ_{s_1,s_2}\\
\hline
\wp & 2 & (0,0) & \wp\\
\hline
\dz\wp & 3 & (0,0) & \dz\wp\\
\hline
\eisen_4 & 4 & (0,0) & \eisen_4\\
\hline
\Eisen_1 & 1 & (0,1) & 2\ic\pi\\
\hline
\eisen_2 & 2 & (1,0) & -2\ic\pi\\
\hline
\end{tabular}
\caption{Fundamental examples of singular quasi-Jacobi forms.}
\label{tab_explefonda}
\end{center}
\end{table}

\begin{lem}\label{lem_indepalgcinq}
The functions \(\wp\), \(\dz\wp\), \(\eisen_4\), \(\Eisen_1\), and \(\eisen_2\) are algebraically independent.
\end{lem}
\begin{proof}
Thanks to Theorem~\ref{prop_strucJS}, it is enough to show that if \(k\), \(s_1\), and \(s_2\) are integers and if the \(f_{j_1,j_2}\) are singular Jacobi forms of weight \(k-j_1-2j_2\) such that
\begin{equation}\label{eq_uniciprof}
\sum_{j_1=0}^{s_1}\sum_{j_2=0}^{s_2}f_{j_1,j_2}\Eisen_1^{j_1}\eisen_2^{j_2}=0
\end{equation}
then, all the \(f_{j_1,j_2}\) are zero. Suppose by contradiction that one is non-zero, we can assume it is~\(f_{s_1,s_2}\). Then, the left-hand side of~\eqref{eq_uniciprof} has depth \((s_1,s_2)\). By uniqueness of depth, we deduce that \(s_1=s_2=0\) since the right-hand side has zero depth, then all the \(f_{j_1,j_2}\) are zero.
\end{proof}

\subsection{Structure}
Section~\ref{subsec_qjacfonda} shows \(\CC[\wp,\dz\wp,\eisen_4,\Eisen_1,\eisen_2]\subseteq\AlgQJS\). The objective of this section is to show the equality of the two algebras.

The proof is based on the following lemma.
\begin{lem}\label{lem_coefestforme}
Let \(f\) be a singular quasi-Jacobi form of weight \(k\) and depth \((s_1,s_2)\). Then \(\cQ_{s_1,s_2}(f)\) is a singular Jacobi form of weight \(k-2s_1-s_2\).
\end{lem}
\begin{proof}
If \(A\) and \(B\) are two elements of \(\Jac\), we have on the one hand
\begin{equation}\label{eq_dbleacun}
f\vert_k(AB)=\sum_{x=0}^{s_1}\sum_{y=0}^{s_2}\cQ_{x,y}(f)\X(AB)^{x}\Y(AB)^{y}
\end{equation}
and on the other hand
\[
f\vert_k(AB)=\left(f\vert_kA\right)\vert_kB=\sum_{j_1=0}^{s_1}\sum_{j_2=0}^{s_2}\left(\cQ_{j_1,j_2}(f)\vert_{k-2j_1-j_2}B\right)\left(\X(A)\vert_2B\right)^{j_1}\left(\Y(A)\vert_1B\right)^{j_2}.
\]
To transform this latter equality, we use Proposition~\ref{prop_JXYcocycles} to obtain
\begin{multline}\label{eq_dbleacdeux}
f\vert_k(AB)=\sum_{x=0}^{s_1}\sum_{y=0}^{s_2}\left(\sum_{j_1=x}^{s_1}\sum_{j_2=y}^{s_2}\binom{j_1}{x}\binom{j_2}{y}\left(-\X(B)\right)^{j_1-x}\left(-\Y(B)\right)^{j_2-y}\right)\left(\cQ_{j_1,j_2}(f)\vert_{k-2j_1-j_2}B\right)\\\X(AB)^x\Y(AB)^y. 
\end{multline}
Comparing the coefficients of \(\X(AB)^{s_1}\Y(AB)^{s_2}\) in~\eqref{eq_dbleacun} and~\eqref{eq_dbleacdeux}, we find
\[
\cQ_{s_1,s_2}(f)\vert_{k-2s_1-s_2}B=\cQ_{s_1,s_2}(f). 
\]
Since \(\cQ_{s_1,s_2}(f)\) is singular, we deduce that \(\cQ_{s_1,s_2}(f)\) is a singular Jacobi form of weight \(k-2s_1-s_2\).
\end{proof}

\begin{thm}\label{thm_strucQJ}
The algebra of singular quasi-Jacobi forms is generated by the functions \(\wp\),\(\dz\wp\),\(\eisen_4\),\(\Eisen_1\) and \(\eisen_2\). Thus, we have
\[
\AlgQJS=\CC[\wp,\dz\wp,\eisen_4,\Eisen_1,\eisen_2].
\]
\end{thm}
\begin{proof}
We have shown (see Theorem~\ref{prop_strucJS}) that \(\AlgJS=\CC[\wp,\dz\wp,\eisen_4]\). Let \(f\in\QJS{k}{s_1}{s_2}\), and set
\[
g=f-(-1)^{s_1}\left(\frac{1}{2\ic\pi}\right)^{s_1+s_2}\cQ_{s_1,s_2}(f)\eisen_2^{s_1}\Eisen_1^{s_2}.
\]
Then
\begin{enumerate}
\item  \(g\in\QJS{k}{s_1-1}{s_2}+\QJS{k}{s_1}{s_2-1}\);
\item \(\cQ_{s_1,s_2}(f)\in\JS{k-2s_1-s_2}\subset\CC[\wp,\dz\wp,\eisen_4]\) according to Lemma~\ref{lem_coefestforme}, so \(g-f\in\CC[\wp,\dz\wp,\eisen_4,\Eisen_1,\eisen_2]\).
\end{enumerate}
Based on Remark~\ref{jacobisinguliereprofnulle}, by induction on \(s_1+s_2\), we obtain
\[
\forall k\in\N\enspace\forall(s_1,s_2)\in\N^2\qquad \QJS{k}{s_1}{s_2}\subseteq\CC[\wp,\dz\wp,\eisen_4,\Eisen_1,\eisen_2].
\]
According to Lemma~\ref{lem_indepalgcinq}, \(\AlgQJS\) is therefore the polynomial algebra \(\CC[\wp,\dz\wp,\eisen_4,\Eisen_1,\eisen_2]\).
\end{proof}
\subsection{Remarkable subalgebras}\label{sec_sousalgrem}
The results of this section are summarized in Figure~\vref{fig_diagram}. 
\subsubsection{Quasi-Jacobi forms of quasielliptic type}\label{subsec_jacell}
\begin{dfn}
We call a \emph{quasi-Jacobi form of quasielliptic type} of weight~\(k\) and depth~\(s\) any singular quasi-Jacobi form of weight \(k\) and depth \((0,s)\).
\end{dfn}
We denote by \(\mathrm{JS}_k^{0,\leq s}\) the vector space of such forms of depth less than or equal to \(s\). We define \(\displaystyle\AlgQJell = \bigoplus_{k=0}^{\infty} \bigcup_{s \geq 0} \mathrm{JS}_k^{0,\leq s}\), which we will call the set of quasi-Jacobi forms of quasielliptic type in the following.

Thanks to Theorem~\ref{thm_strucQJ}, this is a polynomial algebra:
\[
\AlgQJell=\CC[\wp,\dz\wp,\eisen_4,\Eisen_1].
\]
We have \(\AlgM \subset \AlgJS \subset \AlgQJell \subset \AlgQJS\).

Equation~\eqref{eq_dtaurho} shows that \(\AlgQJell\) is not stable under the \emph{modular derivation}
\[
\dtau=\frac{\pi}{2\ic}\frac{\partial}{\partial\tau}.
\]
According to equations~\eqref{eq_delLaurent} and~\eqref{eq_LaurentEun}, we have
\begin{equation}\label{eq_Euwped}
\frac{\partial\Eisen_1}{\partial z}=-\wp-\eisen_2,
\end{equation}
and therefore \(\AlgQJell\) is not stable under the \emph{elliptic derivation}
\[
\dz=\frac{\partial}{\partial z}.
\]

Table~\vref{tab_stabderalginter} summarizes the stability of the various algebras involved under the various derivations with introduced.
 
\subsubsection{Quasi-Jacobi forms of quasimodular type}\label{sec_fqjm}
\begin{dfn}
We call a \emph{quasi-Jacobi form of quasimodular type} of weight \(k\) and depth \(s\) any singular quasi-Jacobi form of weight \(k\) and depth \((s,0)\).
\end{dfn}
We denote by \(\mathrm{JS}_k^{\leq s, 0}\) the vector space of such forms of depth less than or equal to \(s\). We define \(\displaystyle\AlgQJmod = \bigoplus_{k=0}^{\infty} \bigcup_{s \geq 0} \mathrm{JS}_k^{\leq s, 0}\), which we will call the set of quasi-Jacobi forms of quasimodular type in the following.

Thanks to Theorem~\ref{thm_strucQJ}, this is a polynomial algebra:
\[
\AlgQJmod=\CC[\wp,\dz\wp,\eisen_4,\eisen_2].
\]
We have \(\AlgM \subset \AlgJS \subset \AlgQJmod \subset \AlgQJS\) and \(\AlgM \subset \AlgQM \subset \AlgQJmod \subset \AlgQJS\).

By Remark~\ref{rem_profdz}, the algebra \(\AlgQJmod\) is stable under the derivation \(\dz\). Equation~\eqref{eq_dtaurho} shows that it is not stable under the derivation \(\dtau\).

\begin{figure}
\begin{center}
\begin{equation*}
\xymatrix{ & \AlgQJS=\CC[\wp,\dz\wp,\eisen_4,\Eisen_1,\eisen_2] \ar@{-}[rd] \\ 
\AlgQJell=\CC[\wp,\dz\wp,\eisen_4,\Eisen_1] \ar@{-}[ru] \ar@{-}[rd] &&\AlgQJmod=\CC[\wp,\dz\wp,\eisen_4,\eisen_2]   \\ 
&\AlgJS=\CC[\wp,\dz\wp,\eisen_4] \ar@{-}[ru] & \AlgQM=\CC[\eisen_4,\eisen_6,\eisen_2]\ar@{-}[u] \\
& \ar@{-}[u] \AlgM=\CC[\eisen_4,\eisen_6] \ar@{-}[ru]&  }
\end{equation*}
\caption{Remarkable Subalgebras.}
\label{fig_diagram}
\end{center}
\end{figure}
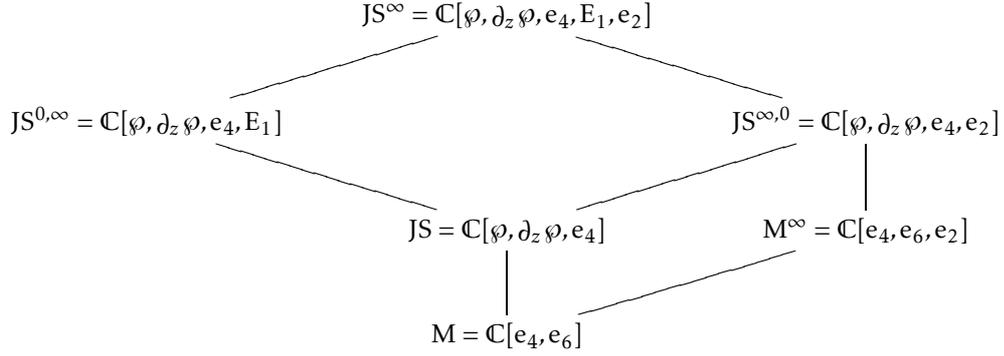

\subsection{Fundamental differential equations} 
\subsubsection{Oberdieck derivation}
\begin{dfn}
We define a derivation on the algebra \(\AlgQJS\) of singular quasi-Jacobi forms by extending the following definition by linearity:  
\[
\text{for any form \(f\in\QJSpoids{k}\)},\qquad\Ob(f)=4\dtau(f)+\Eisen_1\dz(f)-k\eisen_2f. 
\]
We call this derivation the \emph{Oberdieck derivation}. 
\end{dfn}
\begin{rem}
We have \(\Ob=4\pi^2\ObCDMR\) where \(\ObCDMR\) is defined in~\cite{hal03132764}. The name refers to the work of Georg Oberdieck~\cite{oberdieck2014serrederivativeweightjacobi}. The restriction of \(\Ob\) to \(\AlgM\) is the Serre derivation. 
\end{rem}
The derivation \(\Ob\) by definition maps \(\QJS{k}{s_1}{s_2}\) into \(\QJS{k+2}{s_1+1}{s_2+1}\). We have the following more precise proposition: 
\begin{prop}\label{prop_ObstabJS}

\begin{enumerate}
 \item The derivation \(\Ob\) maps \(\QJS{k}{s_1}{s_2}\) into \(\QJS{k+2}{s_1+1}{s_2}\).
 \item The algebra \(\AlgJS\) is stable under \(\Ob\): the image under \(\Ob\) of any singular Jacobi form of weight \(k\) is a singular Jacobi form of weight \(k+2\). 
\end{enumerate}
\end{prop}

\begin{proof}
Let \(f\in\QJSpoids{k}\). Using Proposition~\ref{prop_stabder}, we see that
\begin{multline}\label{eq_Qob}
\cQ_{j_1,j_2}(\Ob(f))=4\dtau(\cQ_{j_1,j_2}(f))+\Eisen_1\dz(\cQ_{j_1,j_2}(f))-k\eisen_2\cQ_{j_1,j_2}(f)+2\ic\pi(j_1+j_2-1)\cQ_{j_1-1,j_2}(f)\\+(j_2+1)\Eisen_1\cQ_{j_1-1,j_2+1}(f). 
\end{multline}

If \(f \in \QJS{k}{s_1}{s_2}\), then \(\cQ_{j_1, s_2+1}(f)=0\) for all \(j_1\), hence \(\Ob(f) \in \QJS{k+2}{s_1+1}{s_2}\).

If \(f\in\JS{k}\), we have \(\cQ_{1,0}\left(\Ob(f)\right)=0\), which shows that \(\Ob(f) \in \JS{k+2}\). 
\end{proof}
\begin{rem}
According to Proposition~\ref{prop_ObstabJS}, the Oberdieck derivation stabilizes \(\AlgQJmod\). However, as we will see in equation~\eqref{eq_ObEun}, it does not stabilize \(\AlgQJell\). 
\end{rem}
\begin{table}\label{table_stabder}
\begin{center}
{\renewcommand{\arraystretch}{1.5}%
\begin{tabular}{|>{$}c<{$}|c|c|c|}
\hline
 & \(\dz\) & \(\dtau\) & \(\Ob\)\\
\hline
\AlgM & yes & no & yes\\
\hline
\AlgJS & yes & no & yes\\
\hline
\AlgQM & yes & yes & yes\\
\hline
\AlgQJell & no & no & no\\
\hline
\AlgQJmod & yes & no & yes\\
\hline
\AlgQJS & yes & yes & yes \\
\hline
\end{tabular}
}
\caption{Stability of algebras under three derivations.}
\label{tab_stabderalginter}
\end{center}
\end{table}

\subsubsection{Applications}\label{subsubsec_application}
The general results from the previous sections allow us, by explicitly calculating the images under derivation of the generators \(\wp,\dz\wp,\eisen_4,\Eisen_1,\eisen_2\), to determine differential relations among these generators. 

The function \(\wp\) is a singular Jacobi form of weight \(2\), and \(\Ob(\wp)\) is therefore a singular Jacobi form of weight \(4\); thanks to Proposition~\ref{prop_dimJS}, the dimension of \(\JS{4}\) is \(2\), with a basis being \((\wp^2,\eisen_4)\). By equating the coefficients of \(1/z^4\) and the constant term, we find
\[
\Ob(\wp)=-2(\wp^2-10\eisen_4). 
\]
From this, we deduce
\begin{equation}\label{eq_dtaurho}
-4\dtau\wp=\Eisen_1\dz\wp+2\wp^2-2\eisen_2\wp-20\eisen_4. 
\end{equation}
Equating the coefficients of \(z^{2n}\) for all \(n\geq 1\) then leads to
\begin{multline}\label{eq_gunther}
2(2n+1)\dtau\eisen_{2n+2}\\=(n+1)(2n+1)\eisen_{2n+2}\eisen_2-(n+2)(2n+5)\eisen_{2n+4}+\sum_{\mathclap{\substack{a \geq 1, b \geq 1\\a+b=n}}}(2a+1)(a-2b-1)\eisen_{2a+2}\eisen_{2b+2}. 
\end{multline}
In particular, for \(n=1\) and \(n=2\) (and considering the equality \(\eisen_8=\frac{3}{7}\eisen_4^2\), which is a consequence of the fact that the space of modular forms of weight \(8\) is of dimension \(1\)), we recover using \eqref{eq_diffwp} the equations of Ramanujan
\begin{subequations}
 \label{eq_ramanujan}
\begin{align}
\dtau\eisen_4 &=\eisen_4\eisen_2-\frac{7}{2}\eisen_6\label{eq_ramae4}\\
&=-\frac{1}{10}\wp^3+\frac{1}{40}\left(\dz\wp\right)^2+\frac{3}{2}\wp\eisen_4+\eisen_4\eisen_2\label{eq_ramae4sanse6}\\
\dtau\eisen_6&=\frac{3}{2}\eisen_6\eisen_2-\frac{15}{7}\eisen_4^2.\label{eq_ramae6} 
\end{align}
\end{subequations}
In particular,
\begin{equation}\label{eq_Obeq}
	\Ob(\eisen_4)=-14\eisen_6=-\frac{2}{5}\wp^3+6\wp\eisen_4+\frac{1}{10}\left(\dz\wp\right)^2.  
\end{equation}	

Thanks to Remark~\ref{rem_profdz}, the function \(\dz^2\wp\) is a singular Jacobi form of weight \(4\) and thus a linear combination of \(\wp^2\) and \(\eisen_4\). By equating the terms in \(z^{-4}\) and the constant terms of the Laurent series expansion, we obtain:
\begin{equation}\label{eq_ddeuxzrho}
\dz^2\wp=6(\wp^2-5\eisen_4). 
\end{equation}

The function \(\dz\wp\) is a singular Jacobi form of weight \(3\), and \(\Ob(\dz\wp)\) is therefore a singular Jacobi form of weight \(5\); the space \(\JS{5}\) has dimension \(1\) spanned by \(\wp\dz\wp\). By equating the coefficients of \(1/z^5\), we find
\[
\Ob(\dz\wp)=-3\wp\dz\wp
\]
from which we deduce
\begin{equation}\label{eq_dtaudzrho}
\dtau\dz\wp=\frac{3}{2}(5\eisen_4-\wp^2)\Eisen_1+\frac{3}{4}(-\wp+\eisen_2)\dz\wp.
\end{equation}

By Proposition~\ref{prop_ObstabJS}, \(\Ob(\Eisen_1)\in\QJS{3}{1}{1}\). We have \(\cQ_{1,1}\left(\Ob(\Eisen_1)\right)=-4\pi^2\), then \(\cQ_{1,0}\left(\Ob(\Eisen_1)\right)=2\ic\pi\Eisen_1=\cQ_{1,0}(-\Eisen_1\eisen_2)\) and \(\cQ_{0,1}\left(\Ob(\Eisen_1)\right)=-2\ic\pi\eisen_2=\cQ_{0,1}(-\Eisen_1\eisen_2)\); we conclude that \(\Ob(\Eisen_1)+\Eisen_1\eisen_2\in\JS{3}=\CC\dz\wp\). Finally,
\begin{equation}\label{eq_ObEun}
\Ob(\Eisen_1)=\frac{1}{2}\dz\wp-\Eisen_1\eisen_2. 
\end{equation}
It follows that \(\AlgQJell\) is not stable under \(\Ob\). Given \eqref{eq_Euwped}, we then obtain
\begin{equation}\label{eq_dtauEun}
4\dtau\Eisen_1=\Eisen_1\eisen_2+\wp\Eisen_1+\frac{1}{2}\dz\wp.
\end{equation}
Similarly, \(\Ob(\eisen_2)\in\QJS{4}{2}{0}\). By \eqref{eq_Qob}, \(\cQ_{2,0}\left(\Ob(\eisen_2)\right)=4\pi^2=\cQ_{2,0}(-\eisen_2^2)\), then \(\cQ_{1,0}\left(\Ob(\eisen_2)\right)=4\ic\pi\eisen_2=\cQ_{1,0}(-\eisen_2^2)\). We deduce that \(\Ob(\eisen_2)+\eisen_2^2\in\JS{4}=\CC\wp^2+\CC\eisen_4\). The \(z\)-dependence shows that \(\Ob(\eisen_2)+\eisen_2^2\in\CC\eisen_4\), and the calculation of the first Fourier coefficient allows us to recover the image of \(\eisen_2\) under the Serre derivation:
\begin{equation}\label{eq_ObEdeux}
 \Ob(\eisen_2) = -\eisen_2^2-5\eisen_4,
\end{equation}
and thus the equation of Ramanujan
\begin{equation}\label{eq_ramaed}
\dtau\eisen_2=\frac{1}{4}\left(\eisen_2^2-5\eisen_4\right).
\end{equation}
\section{Rankin-Cohen brackets and formal deformations}
This section is dedicated to the construction of formal deformations (see \cite[\texten{Chapter}~13]{zbMATH06054532}, \cite[\S~1.1]{zbMATH07362171}) of the various quasi-Jacobi form algebras studied previously.
\subsection{Rankin-Cohen brackets of quasi-Jacobi forms of quasielliptic type}\label{RCquasiell}
According to Proposition~\ref{prop_stabder}, the modular derivation \(\dtau\) of \(\AlgQJS\) is homogeneous of degree \(2\) for this grading: \(\dtau(\QJSpoids{k})\subseteq\QJSpoids{k+2}\) for all \(k\geq 0\). We can then define a formal deformation of \(\AlgQJS\) in the style of formal Rankin-Cohen brackets as defined in \cite{zbMATH07362171}.

\begin{prop}\label{RCtauQ}Consider the sequence \((\crochet{}{}{n})_{n\geq 0}\) of applications from \(\AlgQJS\times \AlgQJS\) to \(\AlgQJS\) defined by bilinear extension of
\begin{equation}\label{RCtau}
\crochet{f}{g}{n}=\sum_{r=0}^n(-1)^r\binom{k+n-1}{n-r}\binom{\ell+n-1}{r}\dtau^r(f)\dtau^{n-r}(g)
\end{equation}
for all \(f\in\QJSpoids{k},g\in\QJSpoids{\ell}\). Then:
\begin{enumerate}[font=\normalfont,label={(\roman*)}]
\item\label{item_unprop} \([\QJSpoids{k},\QJSpoids{\ell}]_n\subseteq\QJSpoids{k+\ell+2n}\) for all \(n,k,\ell\geq 0\).
\item\label{item_deuxprop} The sequence \((\crochet{}{}{n})_{n\geq 0}\) is a formal deformation of \(\AlgQJS\).
\item\label{item_troisprop} The subalgebra \(\AlgM\) is stable under the applications \(\crochet{}{}{n}\), with their restriction coinciding with the classical Rankin-Cohen brackets on modular forms.
\end{enumerate}
\end{prop}
\begin{proof} Points~\ref{item_unprop} and~\ref{item_deuxprop} follow from a direct application of the general algebraic result of \cite[\texten{Proposition}~3]{zbMATH07362171}. Point~\ref{item_troisprop} is the classical result proven, for example, in \cite[\S 5.2]{zbMATH05808162}.
\end{proof}
We have seen in \S~\ref{subsec_jacell} that the subalgebra \(\AlgQJell\) is not stable under the derivation \(\dtau\). However, it is stable under the deformation above.
\begin{thm}\label{RCtauA}
The subalgebra \(\AlgQJell\) is stable under the sequence of Rankin-Cohen brackets \(\left(\crochet{}{}{n}\right)_{n\geq0}\).
\end{thm} 
\begin{proof}
We use the general method of extension-restriction formulated in \texten{Theorem~6} of \cite{zbMATH07362171}. 
We consider the inclusion \(A\subset R\) where we denote \(R=\AlgQJS\) and \(A=\AlgQJell\).
We denote by \(\delta\) the derivation of \(R\) defined by multiplication by half the weight, that is defined by linear extension of
\begin{equation}
\delta(f)=\tfrac{k}{2}f\quad\text{for all \(f\in \QJSpoids{k}.\)}
\end{equation}
We further introduce the derivation of \(R\) defined by
\begin{equation}\label{defDelta}
\theta=\tfrac{1}{4}(\Ob-\Eisen_1\dz)=\dtau-\tfrac{1}{2}e_2\delta.
\end{equation} 
It is clear that \(\delta(A) \subseteq A\). Furthermore, \(A=\AlgJS[\Eisen_1]\), the derivations \(\dz\) and \(\Ob\) stabilize \(\AlgJS\) by Table~\vref{tab_stabderalginter}, hence \(\theta(\AlgJS)\subseteq A\) and
\[
\theta(\Eisen_1)=\frac{1}{8}(\dz\wp+2\wp\Eisen_1)
\] 
thanks to~\eqref{eq_Euwped} and~\eqref{eq_ObEun}. We deduce that \(\theta(A) \subseteq A\).

Moreover, the derivation \(\theta\) is homogeneous of degree 2 for the grading defined by the weight on \(R\) and we have
\begin{equation}\label{Deltatheta}
\delta\theta-\theta\delta=\theta.
\end{equation}
We set \(x=\frac{1}{4}\eisen_2\), which satisfies \(x\in R\) and \(x\notin A\). It satisfies \(\delta(x)=x\) and~\eqref{eq_ObEdeux} shows that \(\theta(x)=-x^2-\frac{5}{16}\eisen_4\). Setting \(h=-\frac{5}{16}\eisen_4\), we have \(h\in A\) with \(\delta(h)=2h\) and \(\theta(x)=-x^2+h\). 

We are thus exactly in the conditions for applying \texten{Theorem~6} of~\cite{zbMATH07362171} with \(\dtau=\theta+2x\delta\), and we conclude that the sequence \((\CM^{\dtau,\delta}_n)_{n\geq 0}\) of Connes-Moscovici brackets associated with the two derivations \(\dtau\) and \(\delta\) defines by restriction to \(A\) a formal deformation of \(A\). These brackets are none other than the Rankin-Cohen brackets \(\left(\crochet{}{}{n}\right)_{n\geq 0}\) as verified by an immediate combinatorial calculation (see, for example, the proof of \texten{Proposition~3} of \cite{zbMATH07362171}).
\end{proof}
\begin{cor} The sequence \((\crochet{}{}{n})_{n\geq 0}\) is a formal deformation of \(\AlgQJell\), which extends the sequence of classical Rankin-Cohen brackets on modular forms.
\end{cor}
\begin{rem}\label{nonstabE} The subalgebras \(\AlgQJmod\) and \(\AlgJS\) are not stable under the brackets \(\crochet{}{}{n}\). For example, it follows from~\eqref{eq_dtaurho} and~\eqref{eq_ramae4} that \(\crochet{\eisen_4}{\wp}{1}\) is of depth \((0,1)\), hence it does not belong to either of these subalgebras. In the following, we construct a formal deformation of \(\AlgJS\) which extends the classical Rankin-Cohen brackets on modular forms.
\end{rem}

\subsection{Rankin-Cohen brackets of singular Jacobi forms}\label{RCsing}
We start by establishing a variant of Proposition~\ref{RCtauQ} by introducing in \(\AlgQJS\) the derivation
\begin{equation}\label{defd}
d=\dtau +\frac{1}{4}\Eisen_1\dz=\frac{1}{4}\Ob+\frac 12 \eisen_2 \delta
\end{equation}
where \(\delta\) is defined by the formula \eqref{defDelta}.
\begin{prop}\label{RCdQ}Consider the sequence \((\Crochet{}{}{n})_{n\geq 0}\) of applications from \(\AlgQJS\times\AlgQJS\) to \(\AlgQJS\) defined by bilinear extension of
\begin{equation}\label{RCd}
\Crochet{f}{g}{n}=\sum_{r=0}^n(-1)^r\binom{k+n-1}{n-r}\binom{\ell+n-1}{r}d^r(f)d^{n-r}(g)
\end{equation}
for all \(f\in\QJSpoids{k},g\in\QJSpoids{\ell}\). Then:
\begin{enumerate}[font=\normalfont,label={(\roman*)}]
\item \(\Crochet{\QJSpoids{k}}{\QJSpoids{\ell}}{n}\subset\QJSpoids{k+\ell+2n}\) for all \(n,k,\ell\geq 0\).
\item The sequence \((\Crochet{}{}{n})_{n\geq 0}\) is a formal deformation of \(\AlgQJS\).
\item The subalgebra \(\AlgM\) is stable under the applications \(\Crochet{}{}{n}\), their restriction coinciding with the classical Rankin-Cohen brackets on modular forms.
\end{enumerate}
\end{prop}
\begin{proof} The derivation \(d\) is homogeneous of degree \(2\). Therefore, it suffices once again to apply \texten{Proposition~3} from \cite{zbMATH07362171}.
\end{proof}
The algebra \(\AlgJS\) is not stable under the derivation \(d\); in fact, it is stable under \(\Ob\) but does not contain \(\eisen_2\). However, it is stable under the above deformation.
\begin{thm}\label{RCtauE}
The subalgebra \(\AlgJS\) is stable under the sequence of Rankin-Cohen brackets \(\left(\Crochet{}{}{n}\right)_{n\geq0}\).
\end{thm}
\begin{proof}
We reuse the structure of the proof of Theorem~\ref{RCtauA}, with \(A\subset R\) for \(R=\AlgQJS\) and \(A=\AlgJS\). This time we introduce the derivation of \(R\) defined by
\(\theta'=\tfrac{1}{4}\Ob\).
According to Proposition~\ref{prop_ObstabJS}, we have \(\delta(A)\subset A\) and \(\theta'(A)\subset A\).

Since \(\theta'\) is homogeneous of degree \(2\), we again have
\begin{equation}\label{Deltathetaprime}
\delta\theta'-\theta'\delta=\theta'.
\end{equation}
The same elements \(x=\frac{1}{4}\eisen_2\) and \(h=-\frac{5}{16}\eisen_4\) satisfy
\[
h\in A,\ \  x\in R, \ \  x\notin A, \ \ \delta(x)=x,\ \ \delta(h)=2h,\ \ \theta'(x)=-x^2+h.
\]

Thus, we conclude in exactly the same way by applying \texten{Theorem 6} from \cite{zbMATH07362171}, this time with \(d=\theta'+2x\delta\), so that the sequence \((\CM^{d,\delta}_n)_{n\geq 0}\) of Connes-Moscovici brackets associated with the two derivations \(d\) and \(\delta\) defines by restriction to \(A\) a formal deformation of \(A\) that coincides with the sequence of Rankin-Cohen brackets \((\Crochet{}{}{n})_{n\geq 0}\) considered here.
\end{proof}
\begin{cor}
The sequence \(\left(\Crochet{}{}{n}\right)_{n\geq0}\) is a formal deformation of \(\AlgJS\), which extends the sequence of classical Rankin-Cohen brackets on modular forms.
\end{cor}
\begin{rem}
The construction of the brackets \eqref{RCd} and the stability of \(\AlgJS\) are demonstrated differently in \cite[\texten{Proposition}~2.15]{zbMATH05953688}.
\end{rem}
\begin{rem}
According to Remark~\ref{nonstabE}, the subalgebra \(\AlgQJmod\) is not stable under \((\crochet{}{}{n})_{n\geq 0}\). However, it is trivially stable under \((\Crochet{}{}{n})_{n\geq 0}\), since \(\AlgQJmod\) is stable under \(\Ob\). It is shown that \(\Crochet{\Eisen_1}{\eisen_4}{1}\) has modular depth \(1\) (for example, using \eqref{eq_Obeq} and \eqref{eq_ObEun}), so that \(\AlgQJell\) is not stable under \(\left(\Crochet{}{}{n}\right)_{n\geq0}\).
\end{rem}
\begin{rem}
 The construction of Rankin-Cohen brackets in Propositions~\ref{RCtauQ} and~\ref{RCdQ} relies on the relations \eqref{Deltatheta} and \eqref{Deltathetaprime} satisfied for the derivations used. A very different construction of a formal deformation of the algebra \(\AlgQJS\) is proposed in what follows, using the derivations \(\dtau\) and~\(\dz\), which satisfy \(\dtau \circ \dz = \dz \circ \dtau\).
\end{rem}

\subsection{Transvectants of quasi-Jacobi forms of quasimodular type}\label{subsec_transvec}
\begin{prop}\label{TransQ}
Consider the sequence \((\accol{}{}{n})_{n\geq 0}\) of bilinear applications from \(\AlgQJS\times\AlgQJS\) to \(\AlgQJS\) defined by
\begin{equation}\label{eq_TransQ}
\accol{f}{g}{n}=\sum_{r=0}^n{(-1)^r}\binom{n}{r}\dtau^{n-r}\dz^{r}(f)\dtau^{r}\dz^{n-r}(g) \quad f,g\in \AlgQJS  
\end{equation}
\begin{enumerate}[font=\normalfont,label={(\roman*)}]
\item\label{item_schubert} The sequence \((\dfrac{1}{n!}\accol{}{}{n})_{n\geq 0}\) is a formal deformation of \(\AlgQJS\).
\item\label{item_beethoven} \(\accol{\QJSpoids{k}}{\QJSpoids{\ell}}{n}\subset \QJSpoids{k+\ell+3n}\) for all \(n,k,\ell\geq 0\).
\end{enumerate}
\end{prop}

\begin{proof}
Point~\ref{item_schubert} is a classical result in invariant theory corresponding to the associativity of the Moyal product (see for example \cite[\texten{Proposition}~5.20]{zbMATH01516969}). 
Point~\ref{item_beethoven} follows from the fact that \(\dtau\) and \(\dz\) are homogeneous of degrees 2 and 1 respectively. 
\end{proof}

\begin{rem}
We recall the following two general properties of transvectants used subsequently. On one hand, they satisfy the recurrence relation:
\begin{equation}\label{rectrans}
\accol{f}{g}{n+1}=\accol{\dtau f}{\dz g}{n}-\accol{\dz f}{\dtau g}{n}
\end{equation}
initialized by the fact that \(\accol{}{}{0}\) is the product in \(\AlgQJS\times\AlgQJS\), and \(\accol{}{}{1}\) is the Poisson bracket \(\dtau\wedge\dz\):
\[
\accol{f}{g}{0}=fg \qquad \text{ and } \qquad \accol{f}{g}{1}=\dtau(f)\dz(g)-\dz(f)\dtau(g).
\]
On the other hand, the associativity of the star product defined on \(\AlgQJS[[\planck]]\) from
\begin{equation}\label{star}
\forall(f,g)\in\AlgQJS\times\AlgQJS \qquad f\star g=\sum_{n\geq 0}\dfrac{1}{n!}\{f,g\}_n\planck^n
\end{equation} 
is equivalent to: 
\begin{equation}\label{asstrans}
\forall(f,g,h)\in\AlgQJS\times\AlgQJS\times\AlgQJS \qquad \sum_{r=0}^n\binom{n}{r}\accol*{\accol{f}{g}{r}}{h}{n-r}=\sum_{r=0}^n\binom{n}{r}\accol*{f}{\accol{g}{h}{r}}{n-r}.
\end{equation}
\end{rem}

We have seen in \S~\ref{sec_fqjm} that \(\AlgQJmod\) is stable under \(\dz\) but not under \(\dtau\). However, it is stable under the transvectants, as we will see below. The proof requires some preliminary technical results.

\begin{lem}\label{lemmeA}
Consider the derivation \(d=\dtau +\tfrac{1}{4}\Eisen_1\dz\) of \(\AlgQJS\); we have:
\begin{enumerate}[font=\normalfont,label={(\roman*)}]
\item \(d(f)\in\AlgQJmod\) and \(\{f, g\}_1 \in \AlgQJmod\) for all \(f,g\in\AlgQJmod\);
\item \(d(E_1)\in\AlgQJmod\) and \(\{f, E_1\}_1 \in \AlgQJmod\) for all \(f\in\AlgQJmod\).
\end{enumerate}
\end{lem}

\begin{proof}
We have already considered in \eqref{defd} the derivation \(d=\dfrac{1}{4} \Ob + \dfrac{1}{2} \eisen_2 \delta\). The algebra \(\AlgQJmod=\AlgJS[\eisen_2]\) is stable under \(\Ob\) according to \S~\ref{subsubsec_application}, and thus it is stable under \(d\). We compute for all \(f,g\in \AlgQJmod\):
\[
\accol{f}{g}{1}=\partial_{\tau}(f)\dz(g)-\dz(f)\partial_{\tau}(g)=d(f)\dz(g)-\dz(f)d(g)\in\AlgQJmod
\]
since \(\AlgQJmod\) is stable under \(d\) and under \(\dz\) according to \S~\ref{sec_fqjm}. 

It follows from~\eqref{eq_ObEun} that
 \begin{equation}\label{eq_dEun}
 d(\Eisen_1)=\tfrac{1}{8}\dz\wp\in\AlgJS\subseteq\AlgQJmod.
\end{equation}
Finally, thanks to \eqref{eq_Euwped}:
\[
\accol{f}{\Eisen_1}{1}=d(f)\dz(\Eisen_1)-d(\Eisen_1)\dz(f)=-(\wp+\eisen_2)d(f)-\frac{1}{8}\dz(f)\dz\wp\in \AlgQJmod. 
\]
\end{proof}

\begin{rem}
For any \(n\in\N\), we have \(d(\Eisen_1^n)=\tfrac{n}{8}(\dz\wp)\Eisen_1^{n-1}\in\AlgQJell=\AlgJS[\Eisen_1]\). However, \(\AlgQJell\) is not stable under \(d\) since, for example, \(d\wp=\tfrac{1}{4}\Ob(\wp)+\tfrac{1}{2}\wp\eisen_2\) with \(\Ob(\wp)\in\AlgJS\) (see Proposition~\ref{prop_ObstabJS}) and \(\wp\eisen_2\notin\AlgQJell\). 
\end{rem}

\begin{lem}\label{lemmeB}
Let \(n \geq 1\) be an integer satisfying the following two properties: 
\begin{enumerate}[font=\normalfont,label={(H\arabic*)}]
 \item\label{item_hypun} for all \(f, g \in \AlgQJmod\), we have \(\accol{f}{g}{n}\in \AlgQJmod\);
 \item\label{item_hypdeux} for all \(f, g \in \AlgQJmod\), we have \(\accol{f\Eisen_1}{g}{n}-\accol{f}{g\Eisen_1}{n}\in\AlgQJmod\). 
\end{enumerate}
Then, for all \(f, g \in \AlgQJmod\), we have \(\accol{f}{g}{n+1} \in\AlgQJmod\) and \(\accol{f}{\Eisen_1}{n+1}\in\AlgQJmod\). 
\end{lem}

\begin{proof}
By the recurrence formula \eqref{rectrans}, we have 
\begin{align*}
\accol{f}{g}{n+1} &=\accol{\dtau f}{\dz g}{n}-\accol{\dz f}{\dtau g}{n}\\
&=-\tfrac{1}{4}\left(\accol{\dz(f)\Eisen_1}{\dz(g)}{n}-\accol{\dz(f)}{\dz(g)\Eisen_1}{n}\right)+\left(\accol{d(f)}{\dz(g)}{n}- \accol{\dz(f)}{d(g)}{n}\right).
\end{align*}

Now, \(\accol{\dz(f)\Eisen_1}{\dz(g)}{n}-\accol{\dz(f)}{\dz(g)\Eisen_1}{n}\in\AlgQJmod\) according to hypothesis~\ref{item_hypdeux} applied to the elements \(\dz(f)\) and \(\dz(g)\) of \(\AlgQJmod\). Similarly, since \(d(f)\) and \(d(g)\) belong to \(\AlgQJmod\) according to Lemma~\ref{lemmeA}, the difference \(\{d(f), \dz(g)\}_n- \{\dz(f), d(g)\}_n\) is also an element of \(\AlgQJmod\) by hypothesis~\ref{item_hypun}. We conclude that \(\accol{f}{g}{n+1}\in\AlgQJmod\).
The same argument applies to \(f \in \AlgQJmod\) and \(g =\Eisen_1\) since \(\dz(\Eisen_1)\) and \(d(\Eisen_1)\) are elements of \(\AlgQJmod\) according to~\eqref{eq_Euwped} and \eqref{eq_dEun}. We thus have \(\accol{f}{\Eisen_1}{n+1}\in\AlgQJmod\), which completes the proof.
\end{proof}

\begin{lem}\label{lemmeC}
For any \(n \geq 1\) and all \(f,g\in\AlgQJS\), we have:
\begin{multline*}
\accol{f\Eisen_1}{g}{n}-\accol{f}{g\Eisen_1}{n}=f\accol{\Eisen_1}{g}{n} +(-1)^{n-1} g\accol{\Eisen_1}{f}{n}\\
-\sum_{i=1}^{n-1}
\binom{n}{i} \left(
\accol*{\accol{f}{\Eisen_1}{i}}{g}{n-i}+(-1)^{n-1}\accol*{\accol{g}{\Eisen_1}{i}}{f}{n-i}
\right).
\end{multline*}
\end{lem}
\begin{proof}
On one hand, we can rewrite each product as a bracket \(\accol{}{}{0}\), on the other hand, for all \(0\leq j\leq n\), the bracket \(\accol{}{}{j}\) is \((-1)^j\)-symmetric. The desired equality can thus be reformulated as
\begin{multline*}
\accol*{\accol{f}{\Eisen_1}{0}}{g}{n}-\accol*{f}{\accol{\Eisen_1}{g}{0}}{n}=\accol*{f}{\accol{\Eisen_1}{g}{n}}{0}-\accol*{\accol{f}{\Eisen_1}{n}}{g}{0}\\
-\sum_{i=1}^{n-1}\binom ni\accol*{\accol{f}{\Eisen_1}{i}}{g}{n-i}+\sum_{i=1}^{n-1}\binom ni\accol*{f}{\accol{\Eisen_1}{g}{i}}{n-i}
\end{multline*}
that is,
\[
\sum_{i=0}^n\binom ni\accol*{\accol{f}{\Eisen_1}{i}}{g}{n-i}=\sum_{i=0}^{n}\binom ni\accol*{f}{\accol{\Eisen_1}{g}{i}}{n-i}.
\]
According to \eqref{star} and \eqref{asstrans}, this identity translates the equality \((f\star\Eisen_1)\star g=f\star(\Eisen_1\star g)\). This last equality holds for all \(f\) and \(g\) in \(\AlgQJS\) due to point~\ref{item_schubert} of Proposition~\ref{TransQ}. 
\end{proof}

\begin{lem}\label{lemmeD}
We have \(\accol{f}{g}{n}\in\AlgQJmod\) and \(\accol{f}{\Eisen_1}{n}\in\AlgQJmod\) for all \(n \geq 1\) and all \(f,g\in \AlgQJmod\).
\end{lem}

\begin{proof} 
We proceed by induction on \(n\). The case \(n=1\) is shown in Lemma~\ref{lemmeA}. If the property is true for all \(1 \leq i \leq n\), Lemma~\ref{lemmeC} then shows that for all \(f, g\in \AlgQJmod\), we have \(\accol{f\Eisen_1}{g}{n}-\accol{f}{g\Eisen_1}{n}\in\AlgQJmod\). We conclude with Lemma~\ref{lemmeB} that \(\accol{f}{g}{n+1}\in\AlgQJmod\) and \(\accol{f}{\Eisen_1}{n+1}\in\AlgQJmod\) for all \(f, g \in \AlgQJmod\).
\end{proof}

We have thus proven that:
\begin{thm}\label{TransB}
The sequence \((\dfrac 1{n!}\accol{}{}{n})_{n\geq 0}\) is a formal deformation of \(\AlgQJmod\).
\end{thm} 
\begin{proof}
This follows immediately from the above lemma and point~\ref{item_schubert} of Proposition~\ref{TransQ}.
\end{proof}

\begin{rem}
The subalgebras \(\AlgQJell\) and \(\AlgJS\) are not stable under \(\left(\accol{}{}{n}\right)_{n\geq0}\) since, for example, \(\{\eisen_4,\wp\}_1\notin\AlgQJell\) according to \eqref{eq_ramae4sanse6}. The brackets \(\accol{}{}{n}\) vanish on \(\mathcal{\AlgM}\) for all \(n\geq 1\). The Poisson structure on \(\AlgQJmod\) defined by the bracket \(\accol{}{}{1}\) is studied in \cite{zhou}. We summarize the situation on page~\pageref{fig_recap}.
\end{rem}

\begin{rem}
With point~\ref{item_beethoven} of Proposition~\ref{TransQ}, Theorem~\ref{TransB} allows us to construct, starting from two quasi-Jacobi forms of quasimodular type with respective weights \(k\) and \(\ell\), a new form in \(\AlgQJmod\) of weight \(k+\ell+3n\), for all \(n \geq 0\). This is a process comparable to that obtained in Sections~\ref{RCquasiell} and~\ref{RCsing} with the Rankin-Cohen brackets on quasi-Jacobi forms of quasielliptic type and on elliptic forms, the increase in weight being \(2n\) in those cases.
\end{rem}

\begin{rem}
The original goal of this study was the construction of formal deformations on an algebra containing elliptic functions obtained from the Weierstra\ss{} function and its derivative. This motivated us to introduce the algebra \(\CC[\wp,\dz\wp,\eisen_4]\) and subsequently \(\CC[\wp,\dz\wp,\eisen_4,\Eisen_1,\eisen_2]\) in order to achieve stability under differentiation. It is then natural to consider the algebras that we have called \emph{quasimodular-type} and \emph{elliptic-type}. A shift in context consists in taking as a starting point the notion of Jacobi forms with possibly nonzero index, possibly defined on a subgroup of the Jacobi group, and then seeking to construct formal deformations in this setting. This work remains to be done.
\end{rem}

\vspace{\fill}
\pagebreak[0]
\vspace{-\fill}

\appendix

\section{Stability of the different algebras under the different brackets}

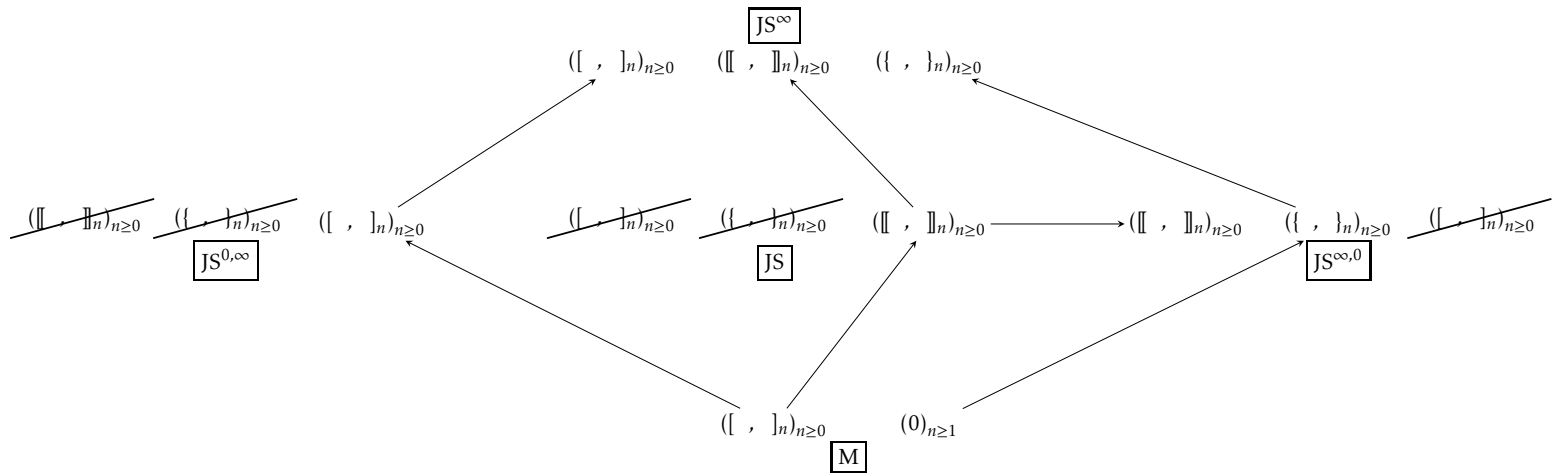
\begin{tikzpicture}
\node[rotate=-90] at (0,0) {%
\begin{minipage}{0.8\textheight}
\scalebox{0.84}{
\begin{tabularx}{\textwidth}{CCCCCCCCCCC}
&&&& \multicolumn{3}{c}{\(\boxed{\AlgQJS}\)} &&&&\\
&&&& \rose{\tikzmark{a}{\(\left(\crochet{}{}{n}\right)_{n\geq0}\)}} & \rose{\tikzmark{c}{\(\left(\Crochet{}{}{n}\right)_{n\geq0}\)}}& \bleu{\tikzmark{i}{\(\left(\accol{}{}{n}\right)_{n\geq0}\)}} &&&&\\
&&&&&&&&&&\\[10ex]
\rose{\(\left(\Crochet{}{}{n}\right)_{n\geq0}\)}  & \bleu{\(\left(\accol{}{}{n}\right)_{n\geq0}\)} & \rose{\tikzmark{b}{\(\left(\crochet{}{}{n}\right)_{n\geq0}\)}} & \phantom{xxxxxx} & \rose{\(\left(\crochet{}{}{n}\right)_{n\geq0}\)} & \bleu{\(\left(\accol{}{}{n}\right)_{n\geq0}\)} & \rose{\tikzmark{d}{\(\left(\Crochet{}{}{n}\right)_{n\geq0}\)}} & \phantom{xxxxxx} & \rose{\tikzmark{e}{\(\left(\Crochet{}{}{n}\right)_{n\geq0}\)}} & \bleu{\tikzmark{h}{\(\left(\accol{}{}{n}\right)_{n\geq0}\)}} & \rose{\(\left(\crochet{}{}{n}\right)_{n\geq0}\)}\\
\multicolumn{3}{c}{\(\boxed{\AlgQJell}\)} && \multicolumn{3}{c}{\(\boxed{\AlgJS}\)} && \multicolumn{3}{c}{\(\quad\boxed{\AlgQJmod}\)} \\
&&&&&&&&&&\\[10ex]
&&&&& \rose{\tikzmark{g}{\(\left(\crochet{}{}{n}\right)_{n\geq0}\)}} & \bleu{\tikzmark{j}{\(\left(0\right)_{n\geq 1}\)}} &&&&\\
&&&&&  \multicolumn{2}{c}{\(\boxed{\AlgM}\)} &&&&
\end{tabularx}
\MeasureLastTable{11} 
\rose{\CrossOut{31}}
\bleu{\CrossOut{32}}
\rose{\CrossOut{35}}
\bleu{\CrossOut{36}}
\rose{\CrossOut{41}}
\label{fig_recap}
\rose{\link{b}{a}}
\rose{\link{g}{b}}
\rose{\link{g}{d}}
\rose{\link{d}{c}}
\rose{\link{d}{e}}
\bleu{\link{h}{i}}
\bleu{\link{j}{h}}
}
\vspace*{10ex}

The arrows indicate extensions; when a bracket is crossed out, it means that it does not stabilize the algebra.
\end{minipage}
};
\end{tikzpicture}

\FloatBarrier

\newpage

\section{Dimensions of the subspaces of quasi-Jacobi forms of index zero}

Let \(k\) be an integer, we define \(\pa(k)=\frac{1+(-1)^k}{2}\) and \(\ia(k)=\frac{1-(-1)^k}{2}\).

\begin{thm}\label{thm_dimsousev}
Let \(k \geq 0\) be an integer. The dimensions \(\fds(k)\) of \(\JS{k}\), \(\fdell(k)\) of \(\JSell{k}\), \(\fdmod(k)\) of \(\JSmod{k}\), and \(\fdq(k)\) of \(\QJSpoids{k}\) are given by
\begin{align*}
\fds(k) &= \frac{107}{288}+\frac{3}{16}k+\frac{1}{48}k^2+\frac{9}{32}(-1)^k+\frac{1}{16}(-1)^kk+\frac{1}{8}\left(\pa(k)+\ia(k)\ic\right)\ic^k+\frac{1}{9}\left(\jc^k+\jc^{2k}\right)\\
\fdell(k) &= \frac{175}{288}+\frac{15}{32}k+\frac{5}{48}k^2+\frac{1}{144}k^3+\frac{5}{32}(-1)^k+\frac{1}{32}(-1)^kk+\frac{1}{8}\pa(k)\ic^k+\frac{1}{27}(1-\jc)\jc^k +\frac{1}{27}(2+\jc)\jc^{2k}\\
\fdmod(k) &=\begin{multlined}[t]\frac{121}{288}+\frac{55}{192}k+\frac{11}{192}k^2+\frac{1}{288}k^3+\frac{13}{32}(-1)^k+\frac{11}{64}(-1)^kk+\frac{1}{64}(-1)^kk^2+\frac{1}{16}\left(\pa(k)+\ia(k)\ic\right)\ic^k\\+\frac{1}{27}(2+\jc)\jc^k +\frac{1}{27}(1-\jc)\jc^{2k}\end{multlined}\\
\fdq(k) &=\begin{multlined}[t]\frac{4267}{6912}+\frac{55}{96}k+\frac{199}{1152}k^2+\frac{1}{48}k^3+\frac{1}{1152}k^4+\frac{63}{256}(-1)^k+\frac{3}{32}(-1)^kk+\frac{1}{128}(-1)^kk^2\\
+\frac{1}{16}\pa(k)\ic^k+\frac{1}{27}\left(\jc^k+\jc^{2k}\right)
\end{multlined}
\end{align*}
where \(\jc=\exp(2\ic\pi/3)\).
\end{thm}

\begin{proof}
Using the same argument as in Proposition~\ref{prop_dimJS}, the generating series for the dimensions are
\begin{align*}
\sum_{k\in\N}\fds(k)\cdot z^k&=\frac{1}{(1-z^2)(1-z^3)(1-z^4)},\\
\sum_{k\in\N}\fdell(k)\cdot z^k&=\frac{1}{(1-z)(1-z^2)(1-z^3)(1-z^4)},\\
\sum_{k\in\N}\fdmod(k)\cdot z^k&=\frac{1}{(1-z^2)^2(1-z^3)(1-z^4)}
\shortintertext{and}
\sum_{k\in\N}\fdq(k)\cdot z^k&=\frac{1}{(1-z)(1-z^2)^2(1-z^3)(1-z^4)}.
\end{align*}

The partial fraction decomposition of the right-hand side justifies that the dimensions are of the form
\[
P_1(k)+P_{-1}(k)(-1)^k+P_{\ic}(k)\ic^k+P_{-\ic}(k)(-\ic)^k+P_{\jc}(k)\jc^k+P_{\jc^2}(k)\jc^{2k}
\]
where the \(P_\xi\) are polynomials whose degree is strictly bounded by the valuation of \(z-\xi\) in the denominator of the generating function (see for example~\cite[\texten{Theorem}~4.4.1]{zbMATH06016068}). These polynomials are easily determined by the beginning of the series expansion. We used \texttt{PARI/GP} for our calculations~\cite{PARI2}.
\end{proof}

From the formulas in Theorem~\ref{thm_dimsousev}, we can derive polynomial formulas with rational coefficients in each class of weight modulo \(12\). Such formulas allow us to obtain ``compact'' expressions for the dimensions similar to equality~\eqref{eq_dimexplicit} in Proposition~\ref{prop_dimJS}, for instance
\[
\fds^{0,\infty}(k)=\projZ*{
\dfrac{1}{144}
\left(k^3+15k^2+\begin{cases}72k+144 & \text{if \(k\) is even}\\ 63k+65 & \text{otherwise}\end{cases}\right)
}.
\]
However, such a formula is somewhat artificial, particularly because it is not unique in its form. For example, we also have
\[
\fds^{0,\infty}(k)=\projZ*{
\frac{k+3}{144}\begin{cases}(k+6)^2 & \text{if \(k\) is even}\\ (k+3)(k+9) & \text{otherwise}\end{cases}
}.
\]


\bibliographystyle{plain}
\bibliography{2024_DuMaRo}

\begin{thebibliography}{10}

\bibitem{zbMATH03646988}
George~E. Andrews.
\newblock A note on partitions and triangles with integer sides.
\newblock {\em Am. Math. Mon.}, 86:477--478, 1979.

\bibitem{zbMATH05156388}
Pierre Bieliavsky, Xiang Tang, and Yijun Yao.
\newblock Rankin-{Cohen} brackets and formal quantization.
\newblock {\em Adv. Math.}, 212(1):293--314, 2007.

\bibitem{zbMATH06071095}
Donald~J. Bindner and Martin Erickson.
\newblock Alcuin's sequence.
\newblock {\em Am. Math. Mon.}, 119(2):115--121, 2012.

\bibitem{zbMATH06696479}
Pierre Charollois and Robert Sczech.
\newblock Elliptic functions according to {Eisenstein} and {Kronecker}: an
  update.
\newblock {\em Eur. Math. Soc. Newsl.}, 101:8--14, 2016.

\bibitem{hal03132764}
Youngju Choie, Fran{\c c}ois Dumas, Fran{\c c}ois Martin, and Emmanuel Royer.
\newblock {A derivation on Jacobi forms: Oberdieck derivation}.
\newblock Disponible sur le serveur d'archive Hal :
  \url{https://hal.science/hal-03132764}, February 2021.

\bibitem{zbMATH07362171}
YoungJu Choie, Fran{\c{c}}ois Dumas, Fran{\c{c}}ois Martin, and Emmanuel Royer.
\newblock Formal deformations of the algebra of {Jacobi} forms and
  {Rankin}-{Cohen} brackets.
\newblock {\em C. R., Math., Acad. Sci. Paris}, 359(4):505--521, 2021.

\bibitem{MR0781735}
Martin Eichler and Don Zagier.
\newblock {\em The theory of {J}acobi forms}, volume~55 of {\em Progress in
  Mathematics}.
\newblock Birkh\"{a}user Boston, Inc., Boston, MA, 1985.

\bibitem{Fogliasso}
Jack Fogliasso.
\newblock {Partial derivatives of Jacobi forms}.
\newblock 35th Automorphic Forms Workshop Louisiana State University, 2023.

\bibitem{zbMATH05500775}
Eberhard Freitag and Rolf Busam.
\newblock {\em Complex analysis.}
\newblock Universitext. Berlin: Springer, 2009.
\newblock 2nd ed.

\bibitem{AlcuinSeq}
OEIS~Foundation Inc.
\newblock Alcuin's sequence: \verb|expansion of x^3/((1-x^2)*(1-x^3)*(1-x^4))|.
\newblock Entry A005044 in The On-Line Encyclopedia of Integer Sequences, 2024.
\newblock \url{https://oeis.org/A005044}.

\bibitem{zbMATH03654914}
J.~H. Jordan, Ray Walch, and R.~J. Wisner.
\newblock Triangles with integer sides.
\newblock {\em Am. Math. Mon.}, 86:686--689, 1979.

\bibitem{zbMATH06054532}
Camille Laurent-Gengoux, Anne Pichereau, and Pol Vanhaecke.
\newblock {\em Poisson structures}, volume 347 of {\em Grundlehren Math. Wiss.}
\newblock Berlin: Springer, 2012.

\bibitem{zbMATH05953688}
Anatoly Libgober.
\newblock Elliptic genera, real algebraic varieties and quasi-{Jacobi} forms.
\newblock In {\em Topology of stratified spaces. Based on lectures given at the
  workshop, Berkeley, CA, USA, September 8--12, 2008}, pages 95--120.
  Cambridge: Cambridge University Press, 2011.

\bibitem{zbMATH05050117}
Fran{\c{c}}ois Martin and Emmanuel Royer.
\newblock Modular forms and periods.
\newblock In {\em Formes modulaires et transcendance. Colloque jeunes}, pages
  1--117. Paris: Soci{\'e}t{\'e} Math{\'e}matique de France, 2005.

\bibitem{oberdieck2014serrederivativeweightjacobi}
Georg Oberdieck.
\newblock A {S}erre derivative for even weight {J}acobi {F}orms, 2014.
\newblock {https://arxiv.org/abs/1209.5628}.

\bibitem{zbMATH06346312}
Ren{\'e} Olivetto.
\newblock On the {Fourier} coefficients of meromorphic {Jacobi} forms.
\newblock {\em Int. J. Number Theory}, 10(6):1519--1540, 2014.

\bibitem{zbMATH01516969}
Peter~J. Olver.
\newblock {\em Classical invariant theory}, volume~44 of {\em Lond. Math. Soc.
  Stud. Texts}.
\newblock Cambridge: Cambridge University Press, 1999.

\bibitem{zbMATH06128504}
Emmanuel Royer.
\newblock Quasimodular forms: an introduction.
\newblock {\em Ann. Math. Blaise Pascal}, 19(2):297--306, 2012.

\bibitem{MR0498338}
Jean-Pierre Serre.
\newblock {\em Cours d'arithm\'{e}tique}, volume No. 2 of {\em Le
  Math\'{e}maticien}.
\newblock Presses Universitaires de France, Paris, 1977.
\newblock Deuxi\`eme \'{e}dition revue et corrig\'{e}e.

\bibitem{zbMATH06016068}
Richard~P. Stanley.
\newblock {\em Enumerative combinatorics. {Vol}. 1.}, volume~49 of {\em Camb.
  Stud. Adv. Math.}
\newblock Cambridge: Cambridge University Press, 2nd ed. edition, 2012.

\bibitem{PARI2}
{The PARI~Group}, Univ. Bordeaux.
\newblock {\em {PARI/GP version \texttt{2.17.0}}}, 2024.
\newblock available from \url{http://pari.math.u-bordeaux.fr/}.

\bibitem{MR4281261}
Jan-Willem van Ittersum, Georg Oberdieck, and Aaron Pixton.
\newblock Gromov-{W}itten theory of {K}3 surfaces and a {K}aneko-{Z}agier
  equation for {J}acobi forms.
\newblock {\em Selecta Math. (N.S.)}, 27(4):Paper No. 64, 30, 2021.

\bibitem{zbMATH07634708}
Jan-Willem~M. van Ittersum.
\newblock The {Bloch}-{Okounkov} theorem for congruence subgroups and {Taylor}
  coefficients of quasi-{Jacobi} forms.
\newblock {\em Res. Math. Sci.}, 10(1):45, 2023.
\newblock Id/No 5.

\bibitem{zbMATH01236956}
Andr{\'e} Weil.
\newblock {\em Elliptic functions according to {Eisenstein} and {Kronecker}.}
\newblock Class. Math. Berlin: Springer, 1999.
\newblock Reprint of the 1976 edition.

\bibitem{MR1280058}
Don Zagier.
\newblock Modular forms and differential operators.
\newblock volume 104, pages 57--75. 1994.
\newblock K. G. Ramanathan memorial issue.

\bibitem{zbMATH05808162}
Don Zagier.
\newblock Elliptic modular forms and their applications.
\newblock In {\em The 1-2-3 of modular forms. Lectures at a summer school in
  Nordfjordeid, Norway, June 2004}, pages 1--103. Berlin: Springer, 2008.

\bibitem{zhou}
Jie Zhou.
\newblock {Poisson structures in theories of modular forms, elliptic functions,
  and invariant theory}.
\newblock Dynamics in Siberia 2017 Conference,
  \url{http://old.math.nsc.ru/conference/ds/2017/talks/talk_by_Zhou.pdf}.

\end{thebibliography}
\end{document}